\documentclass[a4paper,11pt]{amsart}
\usepackage{graphicx}
\usepackage{amsmath,amssymb,amsthm,amsfonts}
\usepackage{amssymb}
\usepackage[active]{srcltx}

\newtheorem{thm}{Theorem}[section]

\newtheorem{lem}[thm]{Lemma}

\newtheorem{rem}[thm]{Remark}
\newtheorem{defn}[thm]{Definition}
\newcommand{\R}{{\mathbb{R}}}

\newcommand{\D}{\hbox{D}}
\def\bbm[#1]{\mbox{\boldmath $#1$}}
\newcommand{\beq }{\begin{equation}}
\newcommand{\eeq }{\end{equation}}

\newcommand{\n}{\nabla}
\newcommand{\la}{\langle}
\newcommand{\ra}{\rangle}

\begin{document}

\parindent 0pc
\parskip 6pt
\overfullrule=0pt

\title[The strong comparison principle]{Harnack inequalities for quasilinear anisotropic elliptic equations with a first order term}

\author{Domenico Vuono$^{*}$}
\address{$^{*}$Dipartimento di Matematica e Informatica, UNICAL, Ponte Pietro  Bucci 31B, 87036 Arcavacata di Rende, Cosenza, Italy}
\email{domenico.vuono@unical.it}

\date{\today}

\date{\today}

\begin{abstract}
We consider weak solutions of the equation 
$$-\Delta_p^H u+a(x,u)H^q(\nabla u)=f(x,u) \quad \text{in } \Omega,$$
where $H$ is in some cases called Finsler norm, $\Omega$ is a domain of $\R^N$, $p>1$, $q\ge \max\{p-1,1\}$, and $a(\cdot,u)$, $f(\cdot,u)$ are functions satisfying suitable assumptions. We exploit the Moser iteration technique to prove a Harnack type  comparison inequality for solutions of the equation and a Harnack type inequality for solutions of the linearized operator. As a consequence, we deduce a Strong Comparison Principle for  solutions of the equation and a strong Maximum Principle for solutions of the linearized operator.

\end{abstract}

	
	.
	
\maketitle

\section{Introduction}
Let $\Omega$ be a domain in $\R^N$, with $N\ge 2$. In this work we consider weak $C^{1,\alpha}_{loc}(\Omega)$ solutions to the equation
\begin{equation}\label{eq:Euler-Lagrange}
-\Delta_p^H u+a(x,u)H^q(\nabla u)=f(x,u) \quad \text{in } \Omega,
\end{equation}
where $p>1$, $q\ge \max\{p-1,1\}$, $a(\cdot,u)$, $f(\cdot,u)$ are functions satisfying suitable assumptions (see $(h_p)$ below) and  
\begin{equation}\label{operatoranisotropic}
\Delta_p^H u:=\operatorname{div}(H^{p-1}(\nabla u)\nabla H(\nabla u))),
\end{equation}
where $\Delta_p^H$ is the so-called anisotropic $p$-Laplacian or Finsler $p$-Laplacian. We point out that $H$ is a Finsler type norm ($H$ satisfies assumptions $(h_H)$, see Section \ref{secpreliminare}). If $H$ is the standard Euclidean norm the operator $\Delta_p^H$ reduces to the classical $p$-Laplace operator. 

We exploit the Moser iteration technique to derive  Harnack inequalities for a class of operators involving the Finsler $p$-Laplace operator. In recent years, considerable attention has been given to qualitative properties of solutions to elliptic problems involving the anisotropic
$p$-Laplace operator. A selection of notable works includes \cite{AGF,anto,BMV,Niremberg,CRS,CianSal,CianSal2,CiFiRo,CiXi1,CiXi2,ERSV,EMSV,FSVuono1,MSP}. Let us mention that anisotropic operators are useful to describe several phenomena
in many applications (such as, for instance, material science
\cite{CaHo, Gu}, biology \cite{AnInMa}, image processing \cite{EsOs,
PeMa}).

Our problem is related to Trudinger's analysis of a degenerate class of operators in weighted Sobolev spaces, as discussed in \cite{Tru}.
In the Euclidean framework, in the papers \cite{DS2,SMS}, the weak and strong Harnack comparison inequality for solutions to the equation \eqref{eq:Euler-Lagrange} are deduced. In the same setting, in \cite{SMS,Montoro}, the authors prove a Harnack-type inequality for solutions to the linearized equation associated with the problem \eqref{eq:Euler-Lagrange} (see Definition \ref{linearize} below). For a more detailed analysis of problem \eqref{eq:Euler-Lagrange} in the case of the pure $p$-Laplacian, we refer to \cite{S2}. 

Here, we prove a Harnack type  comparison inequality for solutions of \eqref{eq:Euler-Lagrange}, a Harnack type inequality for solutions of the linearized operator, and exploit them to prove a Strong Comparison Principle for  solutions of the equation, as well as a strong Maximum Principle for solutions of the linearized operator.

As usual the weak formulation of \eqref{eq:Euler-Lagrange} is 
\begin{equation}\label{eq:debole}
\int_\Omega H^{p-1}(\nabla u) \langle \nabla H(\nabla u), \nabla
\varphi \rangle+a(x,u)H^q(\nabla u)\varphi\,dx\,=\,\int_\Omega f(x,u)\varphi\,dx, \quad\forall \varphi\in
C^1_c(\Omega)\,.
\end{equation}

It is well known that, by standard regularity results (see \cite{BMV,CFV,DB,GiTru,T}), solutions to problems involving the $\Delta_p^H(\cdot)$ operator (and under suitable assumptions) are in general of class $C^{1,\alpha}_{loc}(\Omega)\cap C^2(\Omega\setminus Z_u),$ where $Z_u$ denotes the set where the gradient vanishes.

Here and all the paper we suppose that the functions $a(x,s)$ and $f(x,s)$ satisfy the following set of assumptions
\begin{itemize}
\item[(h$_p$)]
\begin{enumerate}
 \item [(i)]$f(x,\cdot)$ is positive and, more precisely,  $f(x,s)>0$ in $\Omega'$ for every  $\Omega'\subset\subset \Omega$ and for every $s>0$.
 \item[(ii)]  $a(x,\cdot),f(x,\cdot)\in C^1( \Omega' \times [0,+\infty))$ for any $\Omega'\subset\subset\Omega$.
\item [(iii)] $a(x,\cdot)$ and $f(x,\cdot)$ are locally Lipschitz continuous, uniformly w.r.t. $x$. Namely, for every  $\Omega'\subset\subset \Omega$ and for every $M>0$,  there is a  positive constant $L=(M,\Omega')$ such that for every $x \in \Omega'$ and every $ u,v \in [0,M] $ we have:
\begin{equation}
    \vert a(x,u) - a(x,v) \vert \le L \vert u-v \vert, \qquad \vert f(x,u) - f(x,v) \vert \le L \vert u-v \vert. 
\end{equation} 

\end{enumerate}
\end{itemize}

Our first result is stated in the following

\begin{thm}\label{Harnack} 
Suppose that assumptions $(h_H)$ and $(h_p)$ hold.
Let $p>(2N+2)/(N+2)$ and let $u,v\,\in\,C^{1,\alpha}_{loc}(\Omega)$   solutions to  \eqref{eq:Euler-Lagrange} in $\Omega$, with
$q\geq\max\,\{p-1,1\}.$
Suppose that $\overline{B(x_0,6\delta)}\subset\Omega'\subset\subset \Omega$ for some $\delta>0$ and that $$u\leq v \quad \text{in}\quad {B(x_0,6\delta)}.$$
Then there exists  $C=C(p,q,\delta,L, \|v\|_{L^{\infty}(\Omega')},  \|\nabla u\|_{L^{\infty}(\Omega')},  \|\nabla v\|_{L^{\infty}(\Omega')})>0$ such that
\begin{equation}\label{HarnIneq}
\sup_{B(x_0,\delta)}(v-u)\leq C\inf_{B(x_0,2\delta)}(v-u).
\end{equation}
\end{thm}

The iterative technique that we use to prove Theorem \ref{Harnack} is due to Moser \cite{MO}.  It was initially developed by \cite{DEG,Nash} to establish Hölder continuity properties of solutions to certain strictly elliptic linear operators. This method was later extended to study degenerate operators, as seen in \cite{serrin}. Actually the method that we use is the one developed by Trudinger in \cite{Tru} to study a degenerate class of operators in weighted Sobolev spaces. We refer to \cite{DS2,SMS,Montoro} for an appropriate use of this technique in the case of the pure $p$-Laplacian.

An important conseguence of the Harnack comparison inequality is in fact the strong comparison principle for the equation \eqref{eq:Euler-Lagrange}.

\begin{thm}[Strong Comparison Principle]\label{Principio forte}
Let $p>(2N+2)/(N+2)$ and  $u,v\in C^{1,\alpha}_{loc}(\Omega)$ with either $u$ or $v$  weak solution to~\eqref{eq:Euler-Lagrange}. Assume that
$ q\geq\max\,\{p-1,1\}$ and assume that $H$, $f(x,s), a(x,s)$ fulfill  $(h_H)$ and $(h_p)$. Then, if
\begin{equation}\nonumber
-\Delta_p^Hu+a(x,u)H^q(\nabla u)-f(x,u)\leq -\Delta_p^Hv+a(x,v)H^{q}(\nabla v)-f(x,v),\quad u\leq v \quad  \text{in}\quad {\Omega}
\end{equation}
in the weak distributional meaning, it follows that
$$u<v \quad \text{in} \quad \Omega$$
unless $u\equiv v$ in $\Omega$.
\end{thm}

The strong comparison principle for $p$-Laplace equation and for the linearized equation is a complex matter and manly still unsolved. Actually it is not hard to derive the strong comparison principle far from the critical set, see e.g. \cite{FMS,tex2} but it remains an open problem already for $p$-harmonic functions near critical points. Our results significantly depend on examining regions where the gradient of the solutions vanishes. This approach is particularly motivated by its applications, especially in exploring the qualitative properties of solutions. Notably, the works \cite{FMRS,MMIS} highlight how the strong principle can simplify proofs and lead to improved results.

In the last part of this paper, our main effort is to obtain such a Harnack type inequality for solutions to the linearized equation  \eqref{linearizzato} (see Definition \ref{linearize} below), under suitable general assumptions. 

Let us recall the following: 

\begin{defn}[\cite{CES1,CDS,DS1}]
Let $\Omega$ be a bounded smooth domain in $\R^N$. For $\rho\in L^1(\Omega)$, we define the weighted Sobolev space
$H_{\rho}^{1,2}(\Omega)$ as the completion of 
$C^{\infty}(\overline \Omega)$ respect to the norm $$\|\psi\|_{H^{1,2}_{\rho}}=\left(\int_\Omega|\psi|^2\right)^{\frac{1}{2}}+\left(\int_\Omega\rho|\nabla \psi|^2\right)^{\frac{1}{2}}.$$ 

The space $H_{\rho,0}^{1,2}(\Omega)$ is the closure of $C_c^1(\Omega)$ in $H_{\rho}^{1,2}(\Omega).$ 
The space $H_{\rho,loc}^{1,2}(\Omega)$ is defined
accordingly.
\end{defn}

\begin{defn}\label{linearize}
Let $u\in C^{1,\alpha}_{loc}(\Omega)$ be a weak solution to \eqref{eq:Euler-Lagrange} and $v,\varphi\in H_{\rho,loc}^{1,2}(\Omega)$, with support $\varphi \subset \subset \Omega$, with $\rho=|\nabla u|^{p-2}$.     
For any $i=1,...,n$, the linearized operator $L_u^i$ is defined by 
\begin{equation}\label{linearizzato}
\begin{split}
     L^i_u(v,\varphi)   :=(p-1)&\int_\Omega  H^{p-2}(\nabla u) \langle \nabla H(\nabla u),\nabla v\rangle \langle \nabla H(\nabla u),\nabla \varphi \rangle \\ +&\int_\Omega H^{p-1}(\nabla u)\langle D^2 H(\nabla u)\nabla v,\nabla \varphi \rangle + q\int_\Omega  H^{q-1}(\nabla u)\langle \nabla H(\nabla u),\nabla v\rangle a(x,u) \varphi \\ +&\int_\Omega \left(a_{x_i}(x,u)H^q(\nabla u)-f_{x_i}(x,u)\right)\varphi +\int_\Omega \left(a_{s}(x,u)H^q(\nabla u)-f_s(x,u)\right)v\varphi.
    \end{split}
\end{equation}

Moreover $v\in H^{1,2}_{\rho,loc}(\Omega)$ is a weak solution of the $i-$linearized operator if 
\begin{equation}\label{soluzionelinearizzato}
L_u^i(v,\varphi)=0
\end{equation}

for any $\varphi\in H_{\rho}^{1,2}(\Omega)$, with support $\varphi\subset\subset \Omega$.
More generally, $v\in H_{\rho,loc}^{1,2}(\Omega)$ is a weak supersolution (subsolution) of \eqref{soluzionelinearizzato} if $L_u(v,\varphi)\ge 0 (\le 0)$ for any nonnegative $\varphi\in H_\rho^{1,2}(\Omega)
,$ with support $\varphi\subset\subset \Omega$.

\end{defn}

\begin{rem}
    We emphasize that the quantities involved in the definition of supersolution (subsolution) might not be well-defined on the critical set $Z_u$, in the case $1<p\le 2$. Theorem \ref{local1} ensures that the Lebesgue measure of $Z_u$
  is zero, implying that the integrals in \eqref{linearizzato} are well-defined over the entire domain $\Omega$. Moreover, Theorem \ref{local1} also ensures that $\rho=|\nabla u|^{p-2}\in L^1(\Omega)$
 under the assumptions stated in the following results.
\end{rem}

We have the following result:
\begin{thm}\label{Harnacklinearizzo}
     Let $p> (2N+2)/(N+2)$ and $u\in C^{1,\alpha}_{loc}(\Omega)$ be a weak solution to~\eqref{eq:Euler-Lagrange}. Assume that
$ q\geq\max\,\{p-1,1\}$ and suppose that $H$, $f(x,s), a(x,s)$ fulfill $(h_H)$ and $(h_p)$.
Suppose that $\overline {B(x_0,6\delta)}\subset \Omega'\subset\subset \Omega$. 
If $v\in H^{1,2}_{\rho,loc}(\Omega)\cap L^{\infty}_{loc}(\Omega)$ is a nonnegative solution of \eqref{soluzionelinearizzato}, then there exist $C=C(p,q,L,H,\delta,\| u\|_{L^{\infty}_{loc}( \Omega)},\|\nabla  u\|_{L^{\infty}_{loc}( \Omega)})>0$ such that 

\begin{equation}\label{Harnacklinearintro}
\sup_{B(x_0,\delta)} v\le C \left(\inf_{B(x_0,2\delta)} v + \|g\|_{L^{\infty}(B(x_0,6\delta))}\right)
\end{equation}

where $g:=\delta ^{p}( f_{x_i}(x, u)- a_{x_i}(x,u)H^q(\nabla u))$.
\end{thm}

As a consequence we have the following result:

\begin{thm}[Strong Maximum Principle]\label{Strong2}
     Let $p> (2N+2)/(N+2)$ and $u\in C^{1,\alpha}_{loc}(\Omega)$ be a  weak solution to~\eqref{eq:Euler-Lagrange}. Assume that
$ q\geq\max\,\{p-1,1\}$ and suppose that $H$, $f(x,s), a(x,s)$ fulfill $(h_H)$ and $(h_p)$. Assume that $a_{x_i}\le 0$ and $f_{x_i}\ge 0$, for any $i=1,...,n$.
 If $v\in H^{1,2}_{\rho,loc}(\Omega)\cap C(\Omega)$ is a nonnegative supersolution of \eqref{soluzionelinearizzato}, then for any connected domain $\Omega '\subset \Omega$, with $v\ge 0$ in $\Omega'$, we have $v>0$ in $\Omega '$, unless $v\equiv 0$ in $\Omega '$.
\end{thm}

In the following we devote Section \ref{secpreliminare} to the introduction of the main geometrical notion
used in the paper and some useful regularity results about solutions to problem \eqref{eq:Euler-Lagrange}. In Section \ref{Sec3} we prove Theorem \ref{Harnack} and Theorem \ref{Principio forte}. In Section \ref{Sec4}  we establish Theorem \ref{Harnacklinearizzo} and Theorem \ref{Strong2}.

\section{Preliminary results}\label{secpreliminare}
\noindent {\bf Notation.} Generic fixed and numerical constants will
be denoted by $C$ (with subscript in some case) and they will be
allowed to vary within a single line or formula. By $|E|$ we will
denote the Lebesgue measure of a measurable set $E$.

The aim of this section is to recall some qualitative properties of the solutions of \eqref{eq:debole} and geometrical tools about anisotropic elliptic operator defined above.

We recall the following:

\begin{thm}\label{local1}
\label{stima hessiano locale}\rm
Let $1<p<\infty$ and  $u\in C^{1,\alpha}_{loc}(\Omega)$ be a  weak solution to \eqref{eq:Euler-Lagrange}, with $H$, $f(x,s)$ and $a(x,s)$ satisfying assumptions $(h_H)$ and $(h_p)$.
We set $u_i:=\partial u/\partial x_i$, for $\beta\in [0,1)$ and $\gamma<N-2$ ($\gamma=0$ if $N=2$),
there holds:

\begin{equation}\label{eq stima hessiano locale componente}
\sup_{y\in \Omega} \int_{\Omega '\setminus Z_u} \frac{|\nabla u|^{p-2-\beta}|\nabla u_i|^2}{|x-y|^{\gamma}} \,dx \le C \quad \forall i=1,...,N,
\end{equation}
for any $\Omega '\subset\subset \Omega$, where $C= C(\Omega',H,\beta, \gamma,p,N,\|u\|_{W^{1,\infty}_{loc}},\|f\|_{W^{1,\infty}_{loc}})$.
Moreover, fix $t\in [0,p-1)$ and $\gamma<N-2$ ($\gamma=0$ if $N=2$), we have 
\begin{equation}\label{stima peso locale}
\sup\limits_{y\in \Omega}\, \int_{\Omega'}\frac{1}{|\nabla  u|^t |x-y|^\gamma} \,dx \leq \dot C\,,
\end{equation}
where $\dot C=\dot C(\Omega',H,t,\gamma,N,p,\|u\|_{W^{1,\infty}_{loc}}, \|f\|_{W^{1,\infty}_{loc}})$.
\end{thm}

We refer to \cite{BMV,CRS,DS1} for a detailed proof. Note that the positivity of $f(x,\cdot)$ actually is needed to obtain the estimate \eqref{stima peso locale}. Moreover by \eqref{stima peso locale}, it follows that the Lebesgue measure of $Z_u$ is zero.

The summability properties of $|\nabla u|^{-1}$ obtained in Theorem \ref{local1} allows to prove a weighted Poincar\'e-Sobolev  type inequality that will be useful in the sequel. For the proof we refer to \cite{DS1,FMS}.

\begin{thm}\label{Poincarepesata}
Let $p\geq2$ and let $u\in C^{1,\alpha}_{loc}(\Omega)$ be a  weak solution of $\eqref{eq:Euler-Lagrange}$, with $H$, $f(x,s)$ and $a(x,s)$ satisfying assumptions $(h_H)$ and $(h_p)$. Then, for any $\Omega'\subset\subset\Omega$, setting $\rho=|\nabla u|^{p-2}$, we have  that  $H^{1,2}_{\rho,0}(\Omega')$ is continuously embedded in $L^{q}(\Omega')$ for  $1\leq q< \bar 2_p$, where
\[
\frac{1}{\bar{ 2}_p}=\frac{1}{2}-\frac{1}{N}+\frac{p-2}{p-1}\,\frac{1}{N}\,.
\]

Consequently, since $\overline{2}_p>2$, for $w\in H_{\rho,0}^{1,2}(\Omega')$, we have 
\begin{equation}\label{poincareweight}
    \|w\|_{L^2(\Omega ')}\le C_{\rho}\left(\int_{\Omega '}\rho|\nabla w|^2\right)^{1/2}
\end{equation}
with $C_\rho=C_\rho(\Omega ')\rightarrow 0$ if $|\Omega '|\rightarrow 0.$

\end{thm}

\begin{rem}\label{zeromeans}
The proof of Theorem \ref{Poincarepesata} is based on potential estimates, see \cite[Section 7]{GiTru}. Since potential estimates are also available for functions with zero mean, we can prove Theorem \ref{Poincarepesata} for functions with zero mean (see \cite[Corollary 2]{FMS}). In particular, for any $w\in H^{1,2}_{\rho}(\Omega')$ such that $\int_{\Omega '}w=0$ and $\Omega'$ is convex, then the estimate \eqref{poincareweight} holds.
\end{rem}
In the work we will use the fact that, since 
$$0\le \frac{1}{(|\nabla u|+|\nabla v|)^t}\le \frac{1}{|\nabla u|^t},$$
from \eqref{stima peso locale}, it follows
$$\int_{\Omega '} \frac{1}{(|\nabla  u|+|\nabla  v|)^t |x-y|^\gamma} \,dx \leq \dot C.$$

Therefore, we also have a weighted Poincar\'e inequality with weight $\rho=(|\nabla u|+|\nabla v|)^{p-2},$ for any $v\in C^{1,\alpha}_{loc}(\Omega),$ for $p\ge 2$. On the other hand, if $1<p<2$, the weighted Poincar\'e inequality follows at once by the classic Sobolev inequality.

\subsection{Finsler norms}

In the following, we introduce the Finsler norm $H$ and we recall technical results  about anisotropic elliptic operator that will be involved in the proof of our main results.

Let $H$ be a function belonging to $C^{2,\alpha}_{loc}(\R^N \setminus \{0\})$.
$H$ is said a ``\emph{Finsler norm}'' if it satisfies the following assumptions:
\begin{itemize}
	\item[$(h_H)$]
\begin{enumerate}
    \item[(i)] $H(\xi)>0 \quad \forall \xi \in \R^N \setminus \{0\}$;

    \item[(ii)] $H(s \xi) = |s| H(\xi) \quad \forall \xi \in \R^N \setminus \{0\}, \, \forall s \in \R$;

    \item[(iii)] $H$ is \emph{uniformly elliptic}, that means the set $\mathcal{B}_1^H:=\{\xi \in \R^N  :  H(\xi) < 1\}$ is \emph{uniformly convex}
    \begin{equation}\label{Hunifconvex}
        \exists \Lambda > 0: \quad \langle \D^2H(\xi)v, v \rangle \geq \Lambda |v|^2 \quad \forall \xi \in \partial \mathcal{B}_1^H, \; \forall v \in \nabla H(\xi)^\bot.
    \end{equation}
\end{enumerate}
\end{itemize}

Moreover, assumption (iii) is equivalent to assume that $\D^2 (H^2)$ is definite positive, i.e. there exist $\lambda>0$ such that, for all $\xi\neq 0$ and $\eta\in\mathbb{R}^N$,
$$
 \langle D^2 H^2(\xi)\eta,\eta\rangle \geq \lambda |\eta|^2.
$$

Since $H$ a norm in $\R^N$, we immediately get that there exists $\alpha_1,\alpha_2>0$ such that:
\begin{equation}\label{H equiv euclidea}
\alpha_1|\xi|\leq H(\xi)\leq \alpha_2|\xi|,\,\quad\forall\, \xi\in\R^N.
\end{equation}

We recall that, since $H$ is a differentiable and $1$-homogeneous function, it holds the Euler's characterization result, i.e.
\begin{equation}\label{eulero}
\la\n H(\xi),\xi\ra=H(\xi) \qquad\forall\,x\in\R^N,
\end{equation}
and
\begin{equation}\label{grad 0 omog}
\nabla H(t\xi)=\hbox{sign}(t)H(\xi) \qquad\forall\,\xi\in\R^N,\, \forall t\in\R.
\end{equation}

Since $H$ is $1$-homogeneous, we have that $\nabla H$ is $0$-homogeneous and it satisfies$$\nabla H (\xi) = \nabla H \left(|\xi|\frac{\xi}{|\xi|}\right)= \nabla H \left(\frac{\xi}{|\xi|}\right) \qquad \forall \xi \in \R^N\setminus \{0\}.$$
Hence, by the previous equality, we infer that there exists a constant $M>0$ such that
\begin{equation}\label{eq:BddH}
	| \nabla H(\xi)| \leq M \qquad \forall \xi \in \R^N\setminus \{0\}.
\end{equation}

Since $H$ is $1$-homogeneous function we get 
\begin{equation}\label{hess -1 omog}
\D^2 H(t\xi)=\frac{1}{|t|}\D^2H(\xi) \qquad\forall\,\xi \in  \R^N\setminus \{0\},\, \forall t\neq 0,
\end{equation}

For the same reasons there exists a constant $\overline M>0$ such that:
\begin{equation}\label{eq:PropFinsler11}
|D^2H(\xi)|\le \frac{\overline M}{|\xi|}\qquad \forall\,\xi\in\R^N\setminus\{0\}, 
\end{equation}
where $|\cdot|$ denotes the usual Euclidean  norm of a matrix, and
\begin{equation}\label{eq:PropFinsler22}
    D^2H(\xi)\xi=0 \qquad \forall\,\xi\in\R^N\setminus\{0\}.
\end{equation}

We refer to \cite{BaChSh, BePa,CianSal,CoFaVa1,CFV} for more details about Finsler geometry.

In the sequel we use the following standard estimates, see \cite[Lemma 2.2]{BMV}, \cite[Proposition 2.1]{EMSV} and \cite[Lemma 4.1]{CRS} (see also \cite{CFV}).

\begin{lem}\label{stimedamascellianisotrope}
There exist positive constants $C_1, C_2,\tilde C_1,\tilde C_2$ depending on $p$ and $H$ such that for any $\xi\in\R^{N}\setminus \{0\}$ and for all $\eta,  \eta' \in \R^N$ such that $|\eta|+ |\eta'|>0$, it holds the following inequalities
\begin{eqnarray}\label{eq:inequalities1}
&&\langle H^{p-1}(\eta)\nabla H(\eta)-H^{p-1}(\eta')\nabla H(\eta'), \eta- \eta' \rangle \geq C_1
(|\eta|+|\eta'|)^{p-2}|\eta-\eta'|^2, \\ \label{eq:inequalities2}
&&|H^{p-1}(\eta)\nabla H(\eta)-H^{p-1}(\eta')\nabla H(\eta')| \leq C_2 (|\eta| + |\eta'|)^{p-2}|\eta-\eta '|,\\
\label{eq:inequalities3}
&&(p-1)H(\xi)^{p-2}\langle \nabla H(\xi),\eta \rangle^{2} +
   H^{p-1}(\xi)  \langle D^{2}H(\xi) \eta, \eta \rangle \ge \Tilde{C}_1 H^{p-2}(\xi)) |\eta|^{2},\\
   \label{eq:inequalities4}
   &&(p-1)H^{p-2}(\xi) \langle \nabla H(\xi), \eta \rangle\langle \nabla H(\xi), \eta' \rangle + H^{p-1}(\xi) \langle D^{2}H(\xi) \eta, \eta' \rangle\le \tilde{C}_2 H(\xi)^{p-2} |\eta| |\eta'|.
\end{eqnarray}
\end{lem}

\section{The Harnack comparison inequality}\label{Sec3}

This section is devoted to prove Theorem \ref{Harnack} and Theorem \ref{Principio forte}.
We show the main steps needed to prove a weak Harnack comparison inequality. In particular we derive the necessary estimates to apply Moser's iterative method. To complete the proof, we refer to \cite{DS2,SMS,Montoro}, where the details of Moser's technique are thoroughly presented, allowing us to avoid redundancy here.

We establish the following result:

\begin{thm}(Weak Harnack Comparison Inequality)\label{t1}
Suppose that assumptions $(h_H)$ and $(h_p)$ hold.
Let $u,v\in C^{1,\alpha}_{loc}(\Omega)$ and assume that either $u$ or $v$ is a weak  solution to \eqref{eq:Euler-Lagrange}, with $ q\geq\max\,\{p-1,1\}$. Assume  that $\overline{B(x_0,6\delta)}\subset\Omega'\subset\subset\Omega$ for some $\delta>0$ and that
\begin{equation}\label{disuguaglianzaipotesi}
-\Delta_p^Hu+a(x,u)H^q(\nabla u)-f(x,u)\leq -\Delta_p^Hv+a(x,v)H^{q}(\nabla v)-f(x,v),\quad u\leq v \quad  \text{in}\quad {B(x_0,6\delta)}.
\end{equation}
Then: 

\

\begin{itemize}
\item {
Case $(a):p\geq 2$.} Define
\[
\frac{1}{\bar{ 2}_p}=\frac{1}{2}-\frac{1}{N}+\frac{p-2}{p-1}\,\frac{1}{N}.
\]
Then, for every
\begin{equation}\nonumber
0<s<\dfrac{\bar 2_p}{2},
\end{equation}
there exists $C>0$ such that
\begin{equation}\label{tt1}
\|(v-u)\|_{L^s(B(x_0,2\delta))}\leq C\inf_{B(x_0,\delta)}(v-u)
\end{equation}
where $C=C(H,p,q,\delta,L, \|v\|_{L^{\infty}(\Omega')},  \|\nabla u\|_{L^{\infty}(\Omega')},  \|\nabla v\|_{L^{\infty}(\Omega')})$.

\

\item  {Case $(b):{(2N+2)}/{(N+2)}<p<2$.} Define
$$\frac{1}{\bar t^\sharp}=\frac{2p-3}{2(p-1)}$$
and let   $2^*$ be  the classical Sobolev exponent $2^*=2N/(N-2)$.
Then, for every
\begin{equation}\nonumber
0<s<\dfrac{2^*}{\bar t^\sharp},
\end{equation}
there exists $C>0$ such that
\begin{equation}\label{tt2}
\|(v-u)\|_{L^s(B(x_0,2\delta))}\leq C\inf_{B(x_0,\delta)}(v-u)
\end{equation}
where $C=C(H,p,q,\delta,L, \|v\|_{L^{\infty}(\Omega')},  \|\nabla u\|_{L^{\infty}(\Omega')},  \|\nabla v\|_{L^{\infty}(\Omega')})$.
\end{itemize}
\end{thm}

The proof is based on Moser's iteration scheme, as described in \cite{MO}. Here, we use the improved approach introduced by Trudinger \cite{Tru}, which relies only on weighted Sobolev inequalities and avoids using the John-Nirenberg Lemma. In the proof, we obtain the two main estimates needed to start Moser's iterative method. For the full details of the technique, we refer to \cite{DS2,SMS,Montoro}.

\begin{proof}
In the sequel, to simplify the notation, we rename the subdomain $\Omega '$ as $\Omega$.

\textbf{First Estimate:} Let us consider the function $w_{\tau}=v-u+\tau$, where $\tau >0$ in $\Omega$, and we define 
\begin{equation}
    \phi_{\tau}=\eta^2w_{\tau}^{\beta}
    \end{equation}
with $\beta<0$ and $\eta\in C_0^1(B(x_0,6\delta))$. So
\begin{equation}\label{gradiente}
    \nabla\phi_{\tau}=2\eta w_{\tau}^{\beta}\nabla\eta+\beta\eta^2 w_{\tau}^{\beta -1}\nabla w_{\tau}.
\end{equation}
Using $\phi_{\tau}$ as positive test function in \eqref{disuguaglianzaipotesi}, we obtain
\begin{equation}\label{primaequazione}
\begin{split}
  \mathcal{J}_1:=&\int_{\Omega}\langle H^{p-1}(\nabla u)\nabla H(\nabla u)-H^{p-1}(\nabla v)\nabla H(\nabla v), \nabla\phi_{\tau}\rangle \,dx\\
\leq &\int_{\Omega}\left( a(x,v)H^{q}(\nabla v)-a(x,u)H^q(\nabla u)\right)\phi_{\tau}\,dx+\int_{\Omega}(f(x,u)-f(x,v))\phi_{\tau}\,dx\\
:=&\mathcal{J}_2+\mathcal{J}_3.
\end{split}
\end{equation}

\noindent {Case $(a)$: $p\geq 2$.}
Now we estimate $\mathcal{J}_1$. By \eqref{gradiente}, we get

\begin{equation}
    \begin{split}
       \mathcal{J}_1&=
        \int_{\Omega}\langle  H^{p-1}(\nabla u)\nabla H(\nabla u)-H^{p-1}(\nabla v)\nabla H(\nabla v), \nabla w_{\tau}\rangle \beta\eta^2w_{\tau}^{\beta -1} \,dx \\ &+2\int_{\Omega}\langle H^{p-1}(\nabla u)\nabla H(\nabla u)-H^{p-1}(\nabla v)\nabla H(\nabla v), \nabla\eta\rangle\eta w_{\tau}^{\beta} \,dx \\
&:=\mathcal{\tilde J}_1+\mathcal{J}_4
    \end{split}
\end{equation}

By Lemma \ref{stimedamascellianisotrope} (in particular see \eqref{eq:inequalities1}), we get  

\begin{equation}\label{secondaequazione}
\begin{split}
    C_1|\beta|&\int_{\Omega} (|\nabla u|+|\nabla v|)^{p-2} |\nabla (u-v)|^2 \eta^2w_{\tau}^{\beta -1}\,dx\\&\le \mathcal{\tilde J}_1\le \mathcal{ J}_2+\mathcal{ J}_3-\mathcal{ J}_4 \\ &= \int_{\Omega}\left( a(x,v)H^{q}(\nabla v)-a(x,u)H^q(\nabla u)\right)\phi_{\tau}\,dx+\int_{\Omega}(f(x,u)-f(x,v))\phi_{\tau}\,dx \\ &-2 \int_{\Omega}\langle H^{p-1}(\nabla u)\nabla H(\nabla u)-H^{p-1}(\nabla v)\nabla H(\nabla v), \nabla\eta\rangle\eta w_{\tau}^{\beta} \,dx. 
    \end{split}
\end{equation}

Using the weighted Young's inequality and Lemma \ref{stimedamascellianisotrope} (see \eqref{eq:inequalities2}), the last term $\mathcal{J}_4$ of the previous inequality becomes

\begin{equation}\label{secondaequazione2}
    \begin{split}
    \mathcal{J}_4 &\le 2C_2 \int_{\Omega} (|\nabla u|+|\nabla v|)^{p-2} |\nabla (u-v)| |\nabla \eta| \eta w_{\tau}^{\beta } \,dx \\ &\le \epsilon  \int_{\Omega} (|\nabla u|+|\nabla v|)^{p-2} |\nabla (u-v)|^2 \eta^2w_{\tau}^{\beta -1} \,dx\\&+\frac{C(p,H)}{\epsilon}\int_{\Omega} (|\nabla u|+|\nabla v|)^{p-2}|\nabla \eta|^2 w_{\tau}^{\beta+1} \,dx, 
    \end{split}
\end{equation}
where $C(p,H)$ is a positive constant.

Now we set $\rho=(|\nabla u|+|\nabla v|)^{p-2}$. For $\epsilon =\frac{C_1|\beta|}{2}$, by \eqref{secondaequazione} and \eqref{secondaequazione2}, we get 

\begin{equation}\label{terzaequazione}
\begin{split}
    \frac{C_1|\beta|}{2}& \int_\Omega \rho |\nabla (u-v)|^2\eta ^2 w_{\tau}^{\beta-1} \,dx  \\\le & \int_{\Omega}\left( a(x,v)H^{q}(\nabla v)-a(x,u)H^q(\nabla u)\right)\phi_{\tau}\,dx\\&+\int_{\Omega}(f(x,u)-f(x,v))\phi_{\tau}\,dx +\frac{C(p,H)}{|\beta|}\int_\Omega \rho |\nabla \eta|^2 w_{\tau}^{\beta+1} \,dx.
    \end{split}
\end{equation}

By assumptions $(h_p)$, since $\tau$ is a positive constant, using  the mean value theorem and by \eqref{H equiv euclidea} and \eqref{eq:BddH}, the first term on the right side of  \eqref{terzaequazione} can be estimated as

\begin{equation}\label{stimaa}
    \begin{split}
       &\int_{\Omega}\left( a(x,v)H^{q}(\nabla v)-a(x,u)H^q(\nabla u)\right)\phi_{\tau}\,dx \\&=\int_{\Omega}a(x,v)(H^{q}(\nabla v)-H^q(\nabla u))\phi_{\tau}\,dx+\int_\Omega (a(x,v)-a(x,u))H^q(\nabla u)\phi_{\tau}\,dx \\ &\le  C(\|v\|_{L^{\infty}(\Omega)})\int_\Omega (H^{q}(\nabla v)-H^q(\nabla u))\eta^2w_{\tau}^{\beta}\\&+ C(q,H,\|v\|_{L^{\infty}(\Omega)},\|\nabla u\|_{L^{\infty}(\Omega)})\int_\Omega (v-u)\phi_{\tau}\,dx \\ &  \le C(q,H,\|v\|_{L^{\infty}(\Omega)})\int_\Omega (|\nabla u|+|\nabla v|)^{q-1}|\nabla (v-u)|\eta^2 w_{\tau}^{\beta}\,dx\\&+C(q,H,\|v\|_{L^{\infty}(\Omega)},\|\nabla u\|_{L^{\infty}(\Omega)})\int_\Omega w_{\tau}\phi_{\tau}\,dx\\&=C(q,H,\|v\|_{L^{\infty}(\Omega)})\int_\Omega (|\nabla u|+|\nabla v|)^{q-1}|\nabla (v-u)|\eta^2 w_{\tau}^{\beta}\frac{(|\nabla u|+|\nabla v|)^{(p-2)/2}}{(|\nabla u|+|\nabla v|)^{(p-2)/2}}\,dx\\&+C(q,H,\|v\|_{L^{\infty}(\Omega)},\|\nabla u\|_{L^{\infty}(\Omega)})\int_\Omega w_{\tau}\phi_{\tau}\,dx
       \\ &\le \epsilon \int_\Omega \rho H(\nabla (u-v))^2\eta ^2 w_{\tau}^{\beta-1} \,dx
       \\ &+\frac{C(p,q,H,\|v\|_{L^{\infty}(\Omega)},\|\nabla u\|_{L^{\infty}(\Omega)},\|\nabla v\|_{L^{\infty}(\Omega)})}{\epsilon}\int_\Omega \eta^2w_{\tau}^{\beta+1}\,dx\\&+C(q,H,\|v\|_{L^{\infty}(\Omega)},\|\nabla u\|_{L^{\infty}(\Omega)})\int_\Omega \eta^2w_{\tau}^{\beta+1} \,dx
    \end{split}
\end{equation}
 
where in the last inequality we have used the weighted Young's inequality and the fact that $q \ge p/2$, for $p\ge 2$.

By assumptions $(h_p)$ and since $\tau$ is a positive constant, the second term on the right-hand side of \eqref{terzaequazione} becomes 

\begin{equation}\label{stimaf}
    \int_\Omega (f(x,u)-f(x,v))\phi_{\tau} \,dx\le L(\|v\|_{L^{\infty}(\Omega)})\int_\Omega \eta^2w_{\tau}^{\beta+1} \,dx.
\end{equation}

By \eqref{stimaa} and \eqref{stimaf}, for $\epsilon=\frac{C_1|\beta|}{4}$, \eqref{terzaequazione} becomes

\begin{equation}\label{quartaequazione}
     \int_\Omega \rho |\nabla w_{\tau}|^2\eta ^2 w_{\tau}^{\beta-1} \,dx\le \frac{\hat C}{|\beta|}\left(1+\frac{1}{|\beta|}\right)\int_\Omega w_{\tau}^{\beta+1}(\eta^2+\rho|\nabla \eta|^2)\,dx,
\end{equation}

where $\hat C={\hat C(p,q,H,L,\|v\|_{L^{\infty}(\Omega)},\|\nabla u\|_{L^{\infty}(\Omega)},\|\nabla v\|_{L^{\infty}(\Omega)})}$ is a positive constant.

Let us now define 
\begin{equation}\label{deinizionewtau}
\tilde w_{\tau}=\begin{cases}
w_{\tau}^{\frac{\beta +1}{2}} & \hbox{if}\;\;\;\beta\neq -1,\\
\log(w_{\tau}) &\hbox{if}\;\;\;\beta =-1
\end{cases}
\end{equation}
and
\begin{equation}\label{definizioner}
r:=\beta +1.
\end{equation}
If $\beta\neq -1$, by \eqref{deinizionewtau}, the estimate \eqref{quartaequazione} becomes 

\begin{equation}\label{quintaequazione}
\begin{split}
&\int_{\Omega}\rho\eta^2|\nabla \tilde w_{\tau}|^2\,dx=\left (\dfrac{\beta +1}{2}\right )^2\int_{\Omega}\rho\eta^2|\nabla  w_{\tau}|^2w_\tau ^{\beta-1}\,dx\\& \le \frac{\hat C}{|\beta|}\left(1+\frac{1}{|\beta|}\right)\left (\dfrac{\beta +1}{2}\right )^2\int_{\Omega}\tilde w_{\tau}^2(\eta ^2+\rho |\nabla \eta|)^2 \,dx.
\end{split}
\end{equation}

Now let us consider $2<\nu< \bar 2_p$.
Since $\eta \tilde w_\tau\in H^{1,2}_{\rho,0}(\Omega)$, from Theorem \ref{Poincarepesata} we have 
\begin{equation}\label{oraprono}
\begin{split}
&\|\eta \tilde{w}_{\tau}\|^2_{L^{\nu}(\Omega)}\le C_\rho\int_\Omega \rho |\nabla (\eta \tilde w_\tau)|^2\,dx\\&\le 2 C_{\rho}\int_{\Omega}\rho\tilde{w}_{\tau}^2|\nabla\eta|^2+\rho\eta^2|\nabla\tilde{w}_{\tau}|^2\,dx.
\end{split}
\end{equation}

Now we observe that the quantity $ \dfrac{1}{|\beta|}\left (1+\dfrac{1}{|\beta|}\right)$ is bounded if $|\beta|\geq \tilde L>0$, and so we include in the constant $\hat C$. Therefore, since $\rho\in L^{\infty}(\Omega)$, for $p\ge 2$, and using \eqref{quintaequazione} in \eqref{oraprono}, we get
\begin{equation}
\label{norma}
\|\eta \tilde{w}_{\tau}\|^2_{L^{\nu}(\Omega)}\le  \hat C r^2\|\tilde{w}_{\tau}(\eta+|\nabla\eta|)\|^2_{L^2(\Omega)},
\end{equation}
up to redefining the constant $\hat C$.

\textbf{Second Estimate: }To use the Moser's iterative scheme we need to show that there exists $r_0 > 0$ and a positive constant $C$ such that 
\begin{equation}\label{secondostep}
    \left(\int_{B(x_0,5/2\delta)} w_\tau^{r_0}\right)^{\frac{1}{r_0}}\le C\left(\int_{B(x_0,5/2\delta)} w_\tau^{-r_0}\right)^{\frac{1}{-r_0}}
\end{equation}

We define 
\begin{equation}
    \tilde w_{\tau}=\log(w_{\tau}).
\end{equation}

Using \eqref{quartaequazione} with $\beta=-1$, we have 

\begin{equation}\label{betaugualemeno1}
    \int_\Omega \rho |\nabla \tilde w_{\tau}|^2\eta ^2  \,dx\le \hat C\int_\Omega (\eta^2+\rho|\nabla \eta|^2)\,dx,
\end{equation}

Supposing $\eta=1$ in $B(x_0,5\delta)$, since $\rho\in L^\infty(\Omega)$ we have

\begin{equation}\label{peso}
    \int_{B(x_0,5\delta)} \rho |\nabla \tilde w_{\tau}|^2\le \hat C,
\end{equation}
where $\hat C={\hat C(p,q,H,L,\|v\|_{L^{\infty}(\Omega)},\|\nabla u\|_{L^{\infty}(\Omega)},\|\nabla v\|_{L^{\infty}(\Omega)})}$ is a positive constant. Replacing $w_{\tau}$ with $w_{\tau}/k$ with $k$ defined by $k:=e^{\frac{1}{|B(x_0,5\delta)|}\int_{B(x_0,5\delta)}\log  w_\tau \,dx},$ we can also suppose that $\tilde w_\tau$ has zero mean on $B(x_0,5\delta)$. Therefore, by the weighted Sobolev inequality, available for functions with zero mean (see Remark \ref{zeromeans}), and by \eqref{peso}, we get 

\begin{equation}
\|\tilde{w}_\tau\|_{L^{\nu}(B(x_0,5\delta))}\le C_\rho \int_\Omega \rho |\nabla \tilde w_\tau|^2\leq C_\rho\hat C.
\end{equation}

We observe that the constant $k$ does not modify the following calculations.

Let us we set 
 \begin{equation}\label{fi}
 \phi=\eta^2\dfrac{1}{w_{\tau}}(|\tilde{w}_\tau|^{\beta}+(2\beta)^{\beta})\qquad \beta\geq 1,
 \end{equation}
 with $\eta\geq 0$ and $\eta\in C_0^1(B(x_0,5\delta))$. So
\begin{equation}\label{derifi}
\nabla\phi=\dfrac{2\eta}{w_{\tau}}(|\tilde{w}_\tau|^{\beta}+(2\beta)^{\beta})\nabla\eta +\eta^2\dfrac{1}{w_{\tau}^2}(\beta sign(\tilde{w}_\tau)|\tilde{w}_\tau|^{\beta -1}-|\tilde{w}_\tau|^{\beta}-(2\beta)^{\beta})\nabla w_{\tau}
\end{equation}

Using $\phi$ as a positive test function in \eqref{disuguaglianzaipotesi}, we get 

\begin{equation}\label{1equazione}
\begin{split}
  \int_{\Omega}&\langle H^{p-1}(\nabla v)\nabla H(\nabla v)-H^{p-1}(\nabla u)\nabla H(\nabla u), \nabla\phi\rangle \,dx\\
+ &\int_{\Omega}\left( a(x,v)H^{q}(\nabla v)-a(x,u)H^q(\nabla u)\right)\phi \,dx\ge\int_{\Omega}(f(x,v)-f(x,u))\phi \,dx.
\end{split}
\end{equation}

Using \eqref{fi} and \eqref{derifi} in  \eqref{1equazione} we have 

\begin{equation}\label{testthen}
 \begin{split} 
 0\ge -\int_{\Omega}\langle H^{p-1}(\nabla v)\nabla H(\nabla v)-H^{p-1}(\nabla u)\nabla H(\nabla u), \nabla\eta\rangle \frac{2\eta}{w_{\tau}}(|w_{\tau}|^{\beta}+(2\beta)^{\beta}) \,dx\\ -\int_{\Omega}\langle H^{p-1}(\nabla v)\nabla H(\nabla v)-H^{p-1}(\nabla u)\nabla H(\nabla u), \nabla w_{\tau}\rangle\times \\ \times \eta^2\dfrac{1}{w_{\tau}^2}(\beta sign(\tilde{w}_\tau)|\tilde{w}_\tau|^{\beta -1}-|\tilde{w}_\tau|^{\beta}-(2\beta)^{\beta}) \,dx \\-\int_{\Omega}\left( a(x,v)H^{q}(\nabla v)-a(x,u)H^q(\nabla u)\right)\eta^2\dfrac{1}{w_{\tau}}(|\tilde{w}_\tau|^{\beta}+(2\beta)^{\beta}) \,dx\\ + \int_{\Omega}(f(x,v)-f(x,u))\eta^2\dfrac{1}{w_{\tau}}(|\tilde{w}_\tau|^{\beta}+(2\beta)^{\beta}) \,dx \\=:-(I_1+I_2+I_3)+I_4.\qquad\qquad\qquad &
 \end{split}
\end{equation}

By Lemma \ref{stimedamascellianisotrope} (in particular see \eqref{eq:inequalities2}), we get

\begin{equation}\label{stimaI_1}
\begin{split}
    I_1&\le C_2\int_{\Omega}(|\nabla u|+|\nabla v|)^{p-2} |\nabla w_{\tau}|\dfrac{2\eta}{w_{\tau}}(|\tilde{w}_\tau|^{\beta}+(2\beta)^{\beta})|\nabla\eta|\,dx
    \\&=C_2\int_{\Omega}\rho |\nabla w_{\tau}|\dfrac{2\eta}{w_{\tau}}(|\tilde{w}_\tau|^{\beta}+(2\beta)^{\beta})|\nabla\eta|\,dx,
    \end{split}
\end{equation}

where we set $\rho:=(|\nabla u|+|\nabla v|)^{p-2}.$
In the sequel, if $\beta\ge 1$, we will make repeated use of the following inequalities

\begin{equation}\label{desi1}
2|\tilde{w}_\tau|^{\beta -1}\leq\dfrac{\beta -1}{\beta}|\tilde{w}_\tau|^{\beta}+\dfrac{1}{\beta}(2\beta)^{\beta}\leq|\tilde{w}_\tau|^{\beta}+(2\beta)^{\beta},
\end{equation}
or
\begin{equation}\label{desi2}
-\beta sign(\tilde{w}_\tau)|\tilde{w}_\tau|^{\beta -1}+|\tilde{w}_\tau|^{\beta}+(2\beta)^{\beta}\geq\beta |\tilde{w}_\tau|^{\beta -1}\ge 0.
\end{equation}

Using Lemma \ref{stimedamascellianisotrope} (in particular see \eqref{eq:inequalities1}) and by \eqref{desi2}, $I_2$ can be estimated in the following way
\begin{equation}\label{stimaI_2}
\begin{split}
    -I_2&\ge C_1 \int_\Omega (|\nabla u|+|\nabla v|)^{p-2} |\nabla w_\tau|^2\frac{\eta^2}{w_\tau^2}(+\beta |\tilde w_\tau|^{\beta-1})\,dx\\&=C_1\int_\Omega \rho |\nabla \tilde w_\tau|^2\eta^2(+\beta |\tilde w_\tau|^{\beta-1})\,dx.
    \end{split}
\end{equation}

For the term $I_3$, we note that 

\begin{equation}\label{terminei3}
\begin{split}
    I_3&=\int_{\Omega}a(x,v)(H^{q}(\nabla v)-H^q(\nabla u))\eta^2\dfrac{1}{w_{\tau}}(|\tilde{w}_\tau|^{\beta}+(2\beta)^{\beta}) \,dx\\&+\int_{\Omega}( a(x,v)-a(x,u))H^q(\nabla u)\eta^2\dfrac{1}{w_{\tau}}(|\tilde{w}_\tau|^{\beta}+(2\beta)^{\beta}) \,dx\\&\qquad\qquad\qquad\qquad:=I_{3,1}+I_{3,2}.
    \end{split}
\end{equation}

Using assumptions $(h_p)$, by \eqref{H equiv euclidea} and the fact that $\tau$ is a positive constant, it follows

\begin{equation}\label{stimaI_3,2}
    I_{3,2}\le C(q,L,H,\|\nabla u\|_{L^{\infty}(\Omega)})\int_\Omega\eta^2(|\tilde{w}_\tau|^{\beta}+(2\beta)^{\beta}) \,dx,
\end{equation}
where $C(q,L,H,\|\nabla u\|_{L^{\infty}(\Omega)})$ is a positive constant.
By assumptions $(h_p)$, from mean value theorem, by \eqref{H equiv euclidea}, \eqref{eq:BddH} and since $q\ge p-1$, we obtain

\begin{equation}
    \begin{split}
    I_{3,1}&\le C(q,H,\|v\|_{L^{\infty}(\Omega)})\int_{\Omega}(|\nabla v|+|\nabla u|)^{q-1}|\nabla w_\tau|\eta^2\dfrac{1}{w_{\tau}}(|\tilde{w}_\tau|^{\beta}+(2\beta)^{\beta})\times \\& \qquad\qquad\qquad\qquad\qquad\qquad\qquad\qquad\qquad\times\frac{(|\nabla v|+|\nabla u|)^{p-2}}{(|\nabla v|+|\nabla u|)^{p-2}} \,dx \\ &\le {C(p,q,H,\|v\|_{L^{\infty}(\Omega)},\|\nabla u\|_{L^{\infty}(\Omega)},\|\nabla v\|_{L^{\infty}(\Omega)})}\int_\Omega \rho |\nabla \tilde w_{\tau}|\eta^2 |\tilde w_{\tau}|^{\beta}\,dx\\&+{C(p,q,H,\|v\|_{L^{\infty}(\Omega)},\|\nabla u\|_{L^{\infty}(\Omega)},\|\nabla v\|_{L^{\infty}(\Omega)})}\int_\Omega \rho |\nabla \tilde w_{\tau}|\eta^2 (2\beta)^{\beta}\,dx
    \end{split}
\end{equation}
where $C(p,q,H,\|v\|_{L^{\infty}(\Omega)},\|\nabla u\|_{L^{\infty}(\Omega)},\|\nabla v\|_{L^{\infty}(\Omega)})$ is a positive constant.

Using the weighted Young inequality and the standard Young inequality, we get 

\begin{equation}\label{stimaI_3,1}
\begin{split}
    I_{3,1}&\le \epsilon C \int_\Omega \rho |\nabla \tilde w_{\tau}|^2\eta^2|\tilde w_{\tau}|^{\beta-1}\,dx +\frac{C}{\epsilon}\int_\Omega \rho \eta^2|\tilde w_{\tau}|^{\beta+1}\,dx \\&+C\int_\Omega \rho |\nabla \tilde w_{\tau}|^2\eta^2 (2\beta)^\beta\,dx+C\int_\Omega \rho \eta^2 (2\beta)^{\beta}\,dx,
    \end{split}
\end{equation}
where $C={C(p,q,H,\|v\|_{L^{\infty}(\Omega)},\|\nabla u\|_{L^{\infty}(\Omega)},\|\nabla v\|_{L^{\infty}(\Omega)})}$ is a positive constant.

Using \eqref{stimaI_3,2} and \eqref{stimaI_3,1} in the estimate \eqref{terminei3}, the term $I_3$ can be estimated as 
\begin{equation}\label{stimaI_3}
\begin{split}
    I_3&\le C\int_\Omega\eta^2(|\tilde{w}_\tau|^{\beta}+(2\beta)^{\beta})+\epsilon C \int_\Omega \rho |\nabla \tilde w_{\tau}|^2\eta^2|\tilde w_{\tau}|^{\beta-1}\,dx +\frac{C}{\epsilon}\int_\Omega \rho \eta^2|\tilde w_{\tau}|^{\beta+1}\,dx \\&+C\int_\Omega \rho |\nabla \tilde w_{\tau}|^2\eta^2 (2\beta)^{\beta}\,dx+C\int_\Omega \rho \eta^2 (2\beta)^{\beta}\,dx,
    \end{split}
\end{equation}
where $C={C(p,q,L,H,\|v\|_{L^{\infty}(\Omega)},\|\nabla u\|_{L^{\infty}(\Omega)},\|\nabla v\|_{L^{\infty}(\Omega)})}$ is a positive constant.

For the  term $I_4$, by assumptions $(h_p)$ and since $\tau$ is a positive constant, it follows 
\begin{equation}\label{stimaI_4}
\begin{split}
I_4&=\int_{\Omega}(f(x,v)-f(x,u))\eta^2\dfrac{1}{w_{\tau}}(|\tilde{w}_\tau|^{\beta}+(2\beta)^{\beta}) \,dx\\&\leq L(\|v\|_{L^{\infty}(\Omega)})\int_{\Omega}\eta^2(|\tilde{w}_\tau|^{\beta}+(2\beta)^{\beta}) \,dx.
\end{split}
\end{equation}

Using the estimates on terms $I_1,I_2,I_3$ and $I_4$ (see  \eqref{stimaI_1}, \eqref{stimaI_2}, \eqref{stimaI_3} and \eqref{stimaI_4}), in the inequality \eqref{testthen} we obtain
\begin{equation}\label{settimaequazione}
    \begin{split}
     &\beta\int_\Omega \rho |\nabla \tilde w_\tau|^2\eta^2 |\tilde w_\tau|^{\beta-1}\\&\le \int_{\Omega}\rho |\nabla w_{\tau}|\dfrac{2\eta}{w_{\tau}}(|\tilde{w}_\tau|^{\beta}+(2\beta)^{\beta})|\nabla\eta| \,dx\\&+\epsilon C \int_\Omega \rho |\nabla \tilde w_{\tau}|^2\eta^2|\tilde w_{\tau}|^{\beta-1}\,dx +\frac{C}{\epsilon}\int_\Omega \rho \eta^2|\tilde w_{\tau}|^{\beta+1}\,dx \\&+C\int_\Omega \rho |\nabla \tilde w_{\tau}|^2\eta^2 (2\beta)^{\beta}\,dx+C\int_\Omega \rho \eta^2 (2\beta)^{\beta}\,dx\\&+C\int_{\Omega}\eta^2(|\tilde{w}|^{\beta}+(2\beta)^{\beta}) \,dx,
    \end{split}
\end{equation}
where $C={C(p,q,H,L,\|v\|_{L^{\infty}(\Omega)},\|\nabla u\|_{L^{\infty}(\Omega)},\|\nabla v\|_{L^{\infty}(\Omega)})}$ is a positive constant.

Using the weighted Young's inequality and the standard Young's inequality for the first term on the right-hand side of \eqref{settimaequazione}, we get

\begin{equation}\label{okay}
    \begin{split}
        &\int_{\Omega}\rho |\nabla w_{\tau}|\dfrac{2\eta}{w_{\tau}}(|\tilde{w}_\tau|^{\beta}+(2\beta)^{\beta})|\nabla\eta| \,dx
        \\& =\int_{\Omega}\rho |\nabla \tilde w_{\tau}|2\eta(|\tilde{w}_\tau|^{\beta}+(2\beta)^{\beta})|\nabla\eta| \,dx 
        \\ &\le \epsilon \int_\Omega \rho \eta^2|\tilde w_\tau|^{\beta-1}|\nabla \tilde w_\tau|^2+\frac{1}{\epsilon}\int_\Omega\rho |\tilde w_\tau|^{\beta+1}|\nabla \eta|^2 \,dx \\&+\int_\Omega \rho (2\beta)^{\beta}|\nabla \eta|^2 \,dx+\int_\Omega \rho \eta^2 (2\beta)^{\beta}|\nabla \tilde w_\tau|^2 \,dx.
    \end{split}
\end{equation}

By \eqref{peso}, \eqref{okay}, for $\epsilon=\epsilon(\beta)$ small (recalling $\beta\ge 1$), and since $\rho\in L^{\infty}(\Omega)$, the inequality \eqref{settimaequazione} becomes 

\begin{eqnarray}\label{caos1}
&&\beta\int_{\Omega}\rho |\nabla \tilde w_{\tau}|^2\eta^2|\tilde{w}_\tau|^{\beta -1}\,dx\\\nonumber
&\leq&C\int_{\Omega}\rho(|\tilde{w}_\tau |^{\beta +1}+(2\beta)^{\beta})|\nabla\eta|^2\,dx+C\int_{\Omega}\eta^2(|\tilde{w}_\tau|^{\beta}+(2\beta)^{\beta}) \,dx\\\nonumber
&+&C\int_{\Omega}\rho\eta^2|\tilde{w}_\tau|^{\beta+1}\,dx+C(2\beta)^{\beta},
\end{eqnarray}
where $C=C(p,q,H,\delta, L,\|v\|_{L^{\infty}(\Omega)},\|\nabla u\|_{L^{\infty}(\Omega)},\|\nabla v\|_{L^{\infty}(\Omega)})$ is a  positive constant.

Now, since the support of $\eta$ depending on $\delta$, we observe that exist a constant $C$, depending on $\delta$ such that 

$$(2\beta)^\beta=C(\delta)\int_{B(x_0,5\delta)}(2\beta)^{\beta}\eta^2\,dx.$$
Therefore, the estimate \eqref{caos1} becomes 
\begin{eqnarray}\label{caos1irrazionale}
&&\int_{\Omega}\rho |\nabla \tilde w_{\tau}|^2\eta^2|\tilde{w}_\tau|^{\beta -1}\,dx\\\nonumber
&\leq&C\left  (\int_{\Omega}\rho(|\tilde{w}_\tau |^{\beta +1}+(2\beta)^{\beta})|\nabla\eta|^2\,dx+\int_{\Omega}\eta^2(|\tilde{w}_\tau|^{\beta}+(2\beta)^{\beta}) \,dx+\int_{\Omega}\rho\eta^2|\tilde{w}_\tau|^{\beta+1}\,dx\right).
\end{eqnarray}

Now, by Young's inequality, we note that 
\begin{equation}\label{notiamo}
    \int_{\Omega}\eta^2(|\tilde{w}_\tau|^{\beta}+(2\beta)^{\beta}) \,dx\le 2\int_{\Omega}\eta^2(|\tilde{w}_\tau|^{\beta+1}+(2\beta)^{\beta}) \,dx.
\end{equation}

By   \eqref{notiamo} and since $\rho\in L^{\infty}(\Omega)$, the inequality  \eqref{caos1irrazionale} can be written as
\begin{equation}\label{barrafael}
\int_{\Omega}\rho|\nabla \tilde w_{\tau}|^2\eta^2|\tilde{w}_\tau|^{\beta -1}\,dx\leq C\int_{\Omega}(|\tilde{w}_\tau|^{\beta +1}+(2\beta)^{\beta})(\eta^2+\rho |\nabla\eta|^2)\,dx,
\end{equation}
where $C=C(p,q,H,\delta, L,\|v\|_{L^{\infty}(\Omega)},\|\nabla u\|_{L^{\infty}(\Omega)},\|\nabla v\|_{L^{\infty}(\Omega)})$ is a positive constant.

We note that \eqref{barrafael} is similar to \eqref{quartaequazione}, except for the extra term $(2\beta)^\beta$. Actually from \eqref{barrafael} follows the inequality \eqref{secondostep}.

Hence the Moser's iterative technique applies: we conclude the proof and we refer to \cite{DS2,SMS,Montoro} for all details.\\

\noindent {  Case $(b)$: ${(2N+2)}/{(N+2)}<p<2$.} 

Arguiung exactly as in case $(a)$, we are able to obtain $\eqref{quartaequazione}$. Using \eqref{deinizionewtau}, we still get \eqref{quintaequazione} if $\beta \neq 1$ or \eqref{betaugualemeno1} if $\beta =-1$.

Since $1<p<2$ and $u,v\in C^1(\overline{\Omega})$, using the classical Sobolev inequality we get
\begin{equation}\label{sobolevclassica}
\begin{split}
\|\eta \tilde{w}_{\tau}\|^2_{L^{2^*}(\Omega)}&\le C_S\int_\Omega |\nabla (\eta \tilde w_\tau)|^2\,dx\\ &\le C_sC(\|\nabla u\|_{L^{\infty}(\Omega)},\|\nabla v\|_{L^{\infty}(\Omega)})\int_\Omega \rho |\nabla (\eta \tilde w_\tau)|^2\,dx,
\end{split}
\end{equation}
where $2^*$ is the classical Sobolev exponent, $C_S$ is the classical Sobolev constant and $\\ C(\|\nabla u\|_{L^{\infty}(\Omega)},\|\nabla v\|_{L^{\infty}(\Omega)})$ is a positive constant. Therefore (see \eqref{norma}) we deduce that 
\begin{equation}
\label{norma2star}
\|\eta \tilde{w}_{\tau}\|^2_{L^{2^*}(\Omega)}\le  \hat C r^2\int_\Omega \tilde w_\tau^2(\eta^2+\rho|\nabla \eta|^2)\,dx.
\end{equation}

Now, from Theorem \ref{local1} (in particular see \eqref{stima peso locale}), we have that $\rho \in L^t(\Omega)$, with $\\t:=(p-1)/(2-p)\theta$, for $0<\theta<1$. Therefore 
\begin{equation}\label{cosecose}
    \|\eta \tilde{w}_{\tau}\|^2_{L^{2^*}(\Omega)}\le  \hat C r^2\|\rho\|_{L^t(\Omega)}\|\tilde w_\tau(\eta+|\nabla \eta|)\|_{L^{t^\#}(\Omega)},
\end{equation}

where $t^\#=2t/(t-1)$. We now set $\chi '=2^*/t^\#$, and in order to apply the Moser's iteration we only need $\chi '>1$. This condition is satisfied if ${(2N+2)}/{(N+2)}<p<2$.

\end{proof}

The iteration technique is easier in this case, and it lets us prove the following theorem

\begin{thm}\label{pro:p2}
Let $u,v\in C^{1,\alpha}_{loc}(\Omega)$ and assume that either $u$ or $v$ is a weak solution to \eqref{eq:Euler-Lagrange}, with
$ q\geq\max\,\{p-1,1\}$ and  $H$, $f(x,u), a(x,u)$ satisfying assumptions $(h_H)$ and $(h_p)$.  Assume  that $\overline{B(x_0,6\delta)}\subset\Omega'\subset \subset\Omega$ for some $\delta>0$ and that
\begin{equation}\label{eq:ekinotiziap}
-\Delta_p^Hv+a(x,v)H^{q}(\nabla v)-f(x,v)\leq -\Delta_p^Hu+a(x,u)H^q(\nabla u)-f(x,u),\quad u\leq v \quad  \text{in}\quad {B(x_0,5\delta)}.
\end{equation}
We distinguish the two cases:

\

\begin{itemize}
\item {Case $(a):p\geq 2$.} For all $s>1$, there exits $C>0$ such that
\begin{equation}\nonumber
\sup_{B(x_0,\delta)}(v-u)\leq C\|v-u\|_{L^s(B(x_0, 2\delta))},
\end{equation}
with $C=C(p,q, H,\delta, L, \|v\|_{L^{\infty}(\Omega')},  \|\nabla u\|_{L^{\infty}(\Omega')},  \|\nabla v\|_{L^{\infty}(\Omega')})$.

\

\item  { Case $(b):{(2N+2)}/{(N+2)}<p<2$.} Define
$${\bar t^\sharp}=\frac{2(p-1)}{2p-3}.$$
Then, for every $s>\bar t^\sharp/2$, there exists $C>0$ such that
\begin{equation}\nonumber
\sup_{B(x_0,\delta)}(v-u)\leq C\|v-u\|_{L^s(B(x_0, 2\delta))},
\end{equation}
with $C=C(p,q, H,\delta, L, \|v\|_{L^{\infty}(\Omega')},  \|\nabla u\|_{L^{\infty}(\Omega')},  \|\nabla v\|_{L^{\infty}(\Omega')})$.
\end{itemize}
\end{thm}

\begin{proof}
In this case, given $w_\tau:=v-u+\tau$, with $\tau$ positive constant, let us define the function $\varphi= \eta^2 w^{\beta}_\tau$ with $\beta>0$ and $\eta\in C^1_c(B(x_0,5\delta)).$ 

For $p\ge 2$, using $\varphi$ as a test function in \eqref{eq:ekinotiziap}, and repeating the same calculations as in the proof of the Theorem \ref{t1}, we get (see \eqref{norma})
\begin{equation}
\label{normaimportant}
\|\eta \tilde{w}_{\tau}\|^2_{L^{\nu}(\Omega)}\le  \hat C r^2\|\tilde{w}_{\tau}(\eta+|\nabla\eta|)\|^2_{L^2(\Omega)}.
\end{equation}

Now, since $\beta >0$, it follows that $r>1$. A classical Moser's iteration yields

\begin{equation}\nonumber
\sup_{B(x_0,\delta)}(v-u)\leq C\|v-u\|_{L^s(B(x_0, 2\delta))},
\end{equation}
where $C=C(p,q, H,\delta, L, \|v\|_{L^{\infty}(\Omega)},  \|\nabla u\|_{L^{\infty}(\Omega)},  \|\nabla v\|_{L^{\infty}(\Omega)})$ is a positve constant.

For $(2N+2)/(N+2)<p<2$, using the test function $\varphi$ in \eqref{eq:ekinotiziap}, following the same reasoning as in the proof of Theorem \ref{t1}, we get (see \eqref{cosecose})
\begin{equation}\label{cosecose2}
    \|\eta \tilde{w}_{\tau}\|^2_{L^{2^*}(\Omega)}\le  \hat C r^2\|\rho\|_{L^t(\Omega)}\|\tilde w_\tau(\eta+|\nabla \eta|)\|_{L^{t^\#}(\Omega)},
\end{equation}
where $t^\#=2(p-1)/(2p-3)$. Iterating \eqref{cosecose2}, for any $s>t^\#/2$, we deduce 
\begin{equation}\nonumber
\sup_{B(x_0,\delta)}(v-u)\leq C\|v-u\|_{L^s(B(x_0, 2\delta))},
\end{equation}
where $C=C(p,q, H,\delta, L, \|v\|_{L^{\infty}(\Omega)},  \|\nabla u\|_{L^{\infty}(\Omega)},  \|\nabla v\|_{L^{\infty}(\Omega)})$ is a positive constant.

\end{proof}

Now it is easy to deduce the proof of Theorem \ref{Harnack}.
\begin{proof}[Proof of Theorem \ref{Harnack}]
The proof readily follows applying Theorem \ref{t1} and Theorem~\ref{pro:p2}.
\end{proof}
\noindent As a corollary of the Harnack comparison inequality we have the
\begin{proof}[Proof of Theorem \ref{Principio forte}]
Let us set $w:v-u$. Define the set $$U_w=\{x\in \Omega\text{ : }w(x)=0\}.$$
By the continuity of $u$ and $v$ we deduce that $U_w$ is a closed set of $\Omega.$ On the other hand, from Theorem \ref{t1}, we have that $U_w$ is open. Then the thesis follows.
\end{proof}

\section{The weak Harnack inequality for the linearized operator}\label{Sec4}

In this section, we prove Theorem \ref{Harnacklinearizzo} and Theorem \ref{Strong2}. We outline the key steps required to establish a weak Harnack inequality for the linearized operator, including the derivation of essential estimates for applying Moser's iterative method. As already mentioned in Section \ref{Sec3}, for the detailed exposition of Moser's technique, we refer to \cite{DS2,SMS,Montoro}, avoiding unnecessary repetition.

\begin{thm}\label{Harnacklinearizzato}
Let $p\ge 2$ and $u\in C^{1,\alpha}_{loc}(\Omega)$ be a  weak solution to~\eqref{eq:Euler-Lagrange}. Assume that
$ q\geq\max\,\{p-1,1\}$ and suppose that $H$, $f(x,s), a(x,s)$ fulfill $(h_H)$ and $(h_p)$.
Assume that $\overline {B(x_0,6\delta)}\subset \Omega$ and put 
$$\frac{1}{\bar{ 2}_p}=\frac{1}{2}-\frac{1}{N}+\frac{p-2}{p-1}\,\frac{1}{N}.$$
If $v\in H^{1,2}_{\rho,loc}(\Omega)\cap L^{\infty}_{loc}(\Omega)$ is a nonnegative supersolution of \eqref{soluzionelinearizzato}, then for every $0<s<\frac{\bar{ 2}_p}{2}$, there exist $C=C(p,q,H,\delta,\| u\|_{L^{\infty}_{loc}( \Omega)},\|\nabla  u\|_{L^{\infty}_{loc}( \Omega)})>0$ such that 

\begin{equation}\label{Harnacklinear}
\delta^{-\frac{N}{s}} \|v\|_{L^s(B(x_0,2\delta))}\le C \left(\inf_{B(x_0,\delta)} v + \|g\|_{L^{\infty}(B(x_0,6\delta))}\right)
\end{equation}

where $g:=\delta ^{p}( f_{x_i}(x, u)- a_{x_i}(x,u)H^q(\nabla u))$.

If $(2N+2)/(N+2)<p<2$ the same result holds for $0<s<2^*/\bar t^\sharp$,
where $$\frac{1}{\bar t^\sharp}=\frac{2p-3}{2(p-1)}$$
and  $2^*$ is  the classical Sobolev exponent $2^*=2N/(N-2)$.
\end{thm}

\begin{proof}
The function $v$ solves $L^i_u(v,\varphi)\ge 0$  for any nonnegative $\varphi\in H_\rho^{1,2}(\Omega)
,$ with support $\varphi\subset\subset \Omega$. We set $y:=\frac{x-x_0}{\delta}$ and $\tilde{\Omega} =\frac{\Omega-x_0}{\delta}$. Since $H$ is a $1$-homogeneous function, by \eqref{grad 0 omog} and \eqref{hess -1 omog}, recalling $\delta >0$ and rescaling \eqref{linearizzato} we get

\begin{equation}\label{supersol}
\begin{split}
    (p-1)&\int_{\tilde{\Omega}} H^{p-2}(\nabla \tilde u)\langle \nabla H(\nabla \tilde u),\nabla \tilde v \rangle \langle \nabla H(\nabla \tilde u),\nabla \varphi \rangle\,dy \\ 
    +&\int_{\tilde \Omega} H^{p-1}(\nabla \tilde u)\langle D^2H(\nabla \tilde u) \nabla \tilde v, \nabla \varphi \rangle\,dy +q \int_{\tilde \Omega} \delta^{p-q}H^{q-1}(\nabla \tilde u) \langle \nabla H(\nabla \tilde u),\nabla \tilde v\rangle \tilde a(y,\tilde u)\varphi\,dy \\
  - &\int_{\tilde \Omega} \left(  \delta^p\tilde f_s(y,\tilde u)-\delta ^{p-q}\tilde a_s(y,\tilde u)H^{q}(\nabla \tilde u)\right)\tilde v \varphi\,dy\\
 & -\int_{\tilde \Omega}\left( \delta ^{p-1}\tilde f_{y_i}(y,\tilde u)-\delta^{p-q-1}\tilde a_{y_i}(y,\tilde u)H^{q}(\nabla \tilde u)\right) \varphi\,dy \ge 0,
    \end{split}
\end{equation}

for any $\varphi \in H_{\rho}^{1,2}(\tilde \Omega),$ compactly supported in $\tilde \Omega$, and where $$ \tilde w(y,\tilde u(y))=w(x_0+\delta y, u(x_0+\delta y))$$ for every function $w$ in $\Omega$.

Now we set

\begin{equation}\label{totequazione}
    \begin{split}
        \overline{a}(y,\tilde u)&:=\delta^{p-q}\tilde a(y,\tilde u)\\
        \overline{f}(y,\tilde u)&:=\delta^{p}\tilde f(y,\tilde u)\\
        \tilde{c}&:= \overline  f_s(y,\tilde u)-\overline a_s(y,\tilde u)H^{q}(\nabla \tilde u)\\
        \overline{g}&:=\delta ^{-1}(\overline f_{y_i}(y,\tilde u)-\overline a_{y_i}(y,\tilde u)H^{q}(\nabla \tilde u)).
    \end{split}
\end{equation}

By \eqref{totequazione}, the inequality \eqref{supersol} becomes

\begin{equation}\label{supersol2}
\begin{split}
    (p-1)&\int_{\tilde{\Omega}} H^{p-2}(\nabla \tilde u)\langle \nabla H(\nabla \tilde u),\nabla \tilde v \rangle \langle \nabla H(\nabla \tilde u),\nabla \varphi \rangle\,dy \\ 
    +&\int_{\tilde \Omega} H^{p-1}(\nabla \tilde u)\langle D^2H(\nabla \tilde u) \nabla \tilde v, \nabla \varphi \rangle\,dy +q \int_{\tilde \Omega} H^{q-1}(\nabla \tilde u) \langle \nabla H(\nabla \tilde u),\nabla \tilde v\rangle \overline a(y,\tilde u)\varphi\,dy \\
  - &\int_{\tilde \Omega} \tilde c \tilde v \varphi\,dy -\int_{\tilde \Omega}\overline g \varphi\,dy \ge 0.
    \end{split}
\end{equation}

For $\varphi \in H_{\rho}^{1,2}(\tilde \Omega),$ we set $K:=\textit{supp } \varphi\subset \subset \tilde \Omega$ and we define the functions $\overline{v}:=\tilde v+ \|\overline g\|_{L^{\infty}(K)}$ and $\overline{c}:=(\tilde c \tilde v+\overline g)\overline v^{-1}$.
With these notations \eqref{supersol2} becomes :

\begin{equation}\label{newlinearizzato}
\begin{split}
    (p-1)&\int_{\tilde{\Omega}} H^{p-2}(\nabla \tilde u)\langle \nabla H(\nabla \tilde u),\nabla \overline v \rangle \langle \nabla H(\nabla \tilde u),\nabla \varphi \rangle \,dy \\ 
    +&\int_{\tilde \Omega} H^{p-1}(\nabla \tilde u)\langle D^2H(\nabla \tilde u) \nabla \overline v, \nabla \varphi \rangle  \,dy +q \int_{\tilde \Omega} H^{q-1}(\nabla \tilde u) \langle \nabla H(\nabla \tilde u),\nabla \overline v\rangle \overline a(y,\tilde u)\varphi \,dy \\
  - &\int_{\tilde \Omega} \overline{c} \overline v \varphi \,dy \ge 0.
    \end{split}
\end{equation}

We note that 
$$|\overline c|=\left|\frac{\tilde c \tilde v+\overline g}{\overline v}\right|\le \left|\frac{\tilde c \tilde v}{\tilde v+ \|\overline g\|_{L^{\infty}}}\right|+\left|\frac{\overline g}{\tilde v+ \|\overline g\|_{L^{\infty}}}\right| \le \|\tilde c\|_{L^{\infty}(K)}+1.$$

Since $\overline{c}\in L^{\infty}_{loc}(\tilde \Omega)$, we can develop an iterative Moser-type scheme \cite{MO} to prove a weak Harnack inequality for $\overline{v}$. For all details of the proof see \cite{CES0,DS2,SMS}, where the iterative Moser-type technique was developed in a similar setting.

Let us first note that in all the proof, we will work in a subdomain $\tilde \Omega '$. To simplify the notation we relabel it as $\tilde \Omega$ in the sequel.

 \textbf{First Estimate:} Let us consider the function  $\overline{v_{\tau}}=\overline v+\tau$ , where $\tau>0$ (at the end we will let $\tau\rightarrow 0$). Moreover, let us define the function $\varphi_\tau=\eta^2 \overline v_{\tau}^\beta$, with $\beta<0$ and $0\le \eta\in C^1_c(B(0,6))$. So
 \begin{equation}
     \nabla \varphi_\tau =2\eta \overline v_{\tau}^{\beta}\nabla \eta +\beta \eta ^2\overline v_{\tau}^{\beta-1}\nabla \overline v_{\tau}
 \end{equation}
Using $\varphi_\tau$ as a test function in \eqref{newlinearizzato} we get

\begin{equation}\label{stimetta}
    \begin{split}
    I_1+\cdot\cdot\cdot+I_6:=&\\
            2(p-1)&\int_{\tilde{\Omega}} H^{p-2}(\nabla \tilde u)\langle \nabla H(\nabla \tilde u),\nabla \overline v \rangle \langle \nabla H(\nabla \tilde u),\nabla \eta \rangle\eta \overline v_\tau^{\beta}\,dy \\ 
            \beta(p-1)&\int_{\tilde{\Omega}} H^{p-2}(\nabla \tilde u)\langle \nabla H(\nabla \tilde u),\nabla \overline v_\tau \rangle \langle \nabla H(\nabla \tilde u),\nabla \overline v_\tau \rangle \eta^2 \overline v_\tau^{\beta-1}\,dy
            \\
    +2&\int_{\tilde \Omega} H^{p-1}(\nabla \tilde u)\langle D^2H(\nabla \tilde u) \nabla \overline v, \nabla \eta \rangle \eta \overline v_\tau ^{\beta}\,dy \\+\beta &\int_{\tilde \Omega} H^{p-1}(\nabla \tilde u)\langle D^2H(\nabla \tilde u) \nabla \overline v_\tau, \nabla \overline v_\tau \rangle \eta ^2 \overline v_\tau ^{\beta-1}\,dy
    \\+q &\int_{\tilde \Omega} H^{q-1}(\nabla \tilde u) \langle \nabla H(\nabla \tilde u),\nabla \overline v\rangle \overline a(y,\tilde u)\eta ^2\overline v_\tau ^{\beta} \,dy\\
  - &\int_{\tilde \Omega} \overline{c} \overline v \eta^2 \overline v_\tau^{\beta} \,dy\ge 0
\end{split}
\end{equation}

\noindent {Case $(a)$: $p\geq 2$.}

We estimate the terms $-(I_2+I_4)$. Using Lemma \ref{stimedamascellianisotrope} (in particular see \eqref{eq:inequalities3}) we deduce 
\begin{equation}\label{stimasinistra}
    -(I_2+I_4)\ge \tilde C_1|\beta| \int_{\tilde\Omega}  H^{p-2}(\nabla \tilde u) |\nabla \overline v_\tau|^2 \eta ^2 \overline v_\tau^{\beta-1}\,dy.
\end{equation}

By \eqref{stimasinistra},  from Lemma \ref{stimedamascellianisotrope} (see \eqref{eq:inequalities4}) applied to terms $I_1+I_3$, by \eqref{eq:BddH} and since $\overline c\in L^{\infty}(\tilde \Omega)$, the estimate \eqref{stimetta} becomes

\begin{equation}\label{equazioneA}
\begin{split}
    &\tilde C_1|\beta| 
    \int_{\tilde\Omega}  H^{p-2}(\nabla \tilde u) |\nabla \overline v_\tau|^2 \eta ^2 \overline v_\tau^{\beta-1}\,dy  \\& \le 2\tilde C_2 \int_{\tilde\Omega} \tilde H^{p-2}(\nabla \tilde u) |\nabla \overline v_\tau| |\nabla \eta|\eta \overline v_\tau^{\beta}\,dy \\ &+C(q,\delta,p,H,\|\tilde u\|_{L^{\infty}(\tilde \Omega)})\int_{\tilde \Omega} H^{q-1}(\nabla \tilde u)\eta ^2 \overline v_\tau ^{\beta}|\nabla \overline v_\tau|\,dy \\ &+C(q,\delta,p,H,\| \tilde u\|_{L^{\infty}(\tilde \Omega)},\|\nabla \tilde u\|_{L^{\infty}(\tilde \Omega)})\int_{\tilde \Omega}\eta ^2 \overline v_\tau ^{\beta+1}\,dy
    \end{split}
\end{equation}

where $C(q,\delta,p,H,\|\tilde u\|_{L^{\infty}(\tilde \Omega)})$ and $C(q,\delta,p,H,\| \tilde u\|_{L^{\infty}(\tilde \Omega)},\|\nabla \tilde u\|_{L^{\infty}(\tilde \Omega)})$ are positive constants.

Now we set $\tilde \rho=H^{p-2}(\nabla \tilde u)$ and we note that our rescaled weight has the same integrability proprieties of the weight $|\nabla u|^{p-2}$.

Since $q\ge p-1$ and using the weighted Young's inequality, \eqref{equazioneA} becomes:

\begin{equation}
\begin{split}
     &\tilde C_1|\beta| 
    \int_{\tilde\Omega} \tilde \rho|\nabla \overline v_\tau|^2 \eta ^2 \overline v_\tau^{\beta-1}\,dy  \\& \le 2\tilde C_2 \int_{\tilde\Omega} \tilde \rho |\nabla \overline v_\tau| |\nabla \eta|\eta \overline v_\tau^{\beta}\,dy \\ &+C(q,\delta,p,H,\|\tilde u\|_{L^{\infty}(\tilde \Omega)})\int_{\tilde \Omega} H^{q-1}(\nabla \tilde u)\eta ^2 \overline v_\tau ^{\beta}|\nabla \overline v_\tau|\frac{H^{(p-2)/2}(\nabla \tilde u)}{H^{(p-2)/2}(\nabla \tilde u)}\,dy \\ &+C(q,\delta,p,H,\| \tilde u\|_{L^{\infty}(\tilde \Omega)},\|\nabla \tilde u\|_{L^{\infty}(\tilde \Omega)})\int_{\tilde \Omega}\eta ^2 \overline v_\tau ^{\beta+1}\,dy \\
    & \le \varepsilon \int_{\tilde\Omega} \tilde \rho |\nabla \overline v_\tau|^2 \eta^2 \overline v_\tau^{\beta-1}\,dy +\frac{C}{\varepsilon} \int_{\tilde\Omega} \tilde \rho  |\nabla \eta|^2 \overline v_\tau^{\beta+1}\,dy \\
    &+\frac{C}{\varepsilon}\int_{\tilde\Omega}  \eta^2 \overline v_\tau^{\beta+1}\,dy,
    \end{split}
\end{equation}
for some positive constant $C=C(q,p,\delta,H,\| \tilde u\|_{L^{\infty}(\tilde \Omega)},\|\nabla \tilde u\|_{L^{\infty}(\tilde \Omega)})$.

For $\epsilon =\tilde C_1|\beta|/2$, if we suppose that $|\beta|\ge c> 0$, up to redefining constants,  we obtain

\begin{equation}\label{stimaimportante}
    \int_{\tilde\Omega} \tilde \rho |\nabla \overline v_\tau|^2 \eta ^2 \overline v_\tau^{\beta-1} \,dy\le \frac{C}{|\beta|}\left(1+\frac{1}{|\beta|}\right)\int_{\tilde\Omega} \overline v_\tau^{\beta+1}(\eta ^2+\tilde \rho |\nabla \eta|^2)\,dy,
\end{equation}

where $C=C(q,p,\delta,H,\| \tilde u\|_{L^{\infty}(\tilde \Omega)},\|\nabla \tilde u\|_{L^{\infty}(\tilde \Omega)})$ is a positive constant. We point out that the estimate \eqref{stimaimportante} is equivalent to the estimate \eqref{quartaequazione}.

Let us now define 
\begin{equation}\label{deinizionewtaulinearizzato}
\tilde w_{\tau}=\begin{cases}
\overline v_{\tau}^{\frac{\beta +1}{2}} & \hbox{if}\;\;\;\beta\neq -1,\\
\log(\overline v_{\tau}) &\hbox{if}\;\;\;\beta =-1
\end{cases}
\end{equation}
and
\begin{equation}\label{definizionerciaociao}
r:=\beta +1.
\end{equation}
If $\beta\neq -1$, by \eqref{deinizionewtaulinearizzato}, the estimate \eqref{stimaimportante} becomes 

\begin{equation}\label{quintaequazionehy}
\begin{split}
&\int_{\tilde \Omega}\tilde \rho\eta^2|\nabla \tilde w_{\tau}|^2\,dy=\left (\dfrac{\beta +1}{2}\right )^2\int_{\tilde \Omega}\tilde \rho\eta^2|\nabla  \overline v_{\tau}|^2\overline v_\tau ^{\beta-1}\,dy\\& \le \frac{C}{|\beta|}\left(1+\frac{1}{|\beta|}\right)\left (\dfrac{\beta +1}{2}\right )^2\int_{\tilde \Omega}\tilde w_{\tau}^2(\eta ^2+\tilde\rho |\nabla \eta|)^2 \,dy.
\end{split}
\end{equation}

Now let us consider $2<\nu< \bar 2_p$.
Since $\eta \tilde w_\tau\in H^{1,2}_{\rho,0}(\tilde \Omega)$, from Theorem \ref{Poincarepesata}, and proceeding as in the proof of Theorem \ref{t1}, we obtain 
\begin{equation}
\label{normab1}
\|\eta \tilde{w}_{\tau}\|^2_{L^{\nu}(\tilde \Omega)}\le  \hat C r^2\|\tilde{w}_{\tau}(\eta+|\nabla\eta|)\|^2_{L^2(\tilde \Omega)},
\end{equation}
up to redefining the constant $\hat C$.

\textbf{Second Estimate:} To carry on the iterative scheme we need to show that there exists $r_0 > 0$ and a positive constant $C$ such that 
\begin{equation}\label{secondstima}
    \left(\int_{B(0,5/2)}\overline v_\tau^{r_0}\right)^{\frac{1}{r_0}}\le C\left(\int_{B(0,5/2)}\overline v_\tau^{-r_0}\right)^{\frac{1}{-r_0}}
\end{equation}

We define 
\begin{equation}
    \tilde w_\tau=\log (\overline v_\tau).
\end{equation}

Using \eqref{stimaimportante} with $\beta=-1$, we have 

\begin{equation}\label{betaugualemeno11}
    \int_\Omega \tilde \rho |\nabla \tilde w_{\tau}|^2\eta ^2  \,dx\le  C\int_\Omega (\eta^2+\tilde \rho|\nabla \eta|^2)\,dx.
\end{equation}

Supposing $\eta=1$ in $B(0,5)$, since $\tilde\rho\in L^\infty(\tilde \Omega)$ we have

\begin{equation}\label{peso11}
    \int_{B(0,5)} \tilde \rho |\nabla \tilde w_{\tau}|^2\le  C,
\end{equation}
where $C=C(q,p,\delta,H,\| \tilde u\|_{L^{\infty}(\tilde \Omega)},\|\nabla \tilde u\|_{L^{\infty}(\tilde \Omega)})$ is a positive constant.

Arguing exactly as in the proof of the Theorem \ref{t1} we can suppose that $\tilde w_\tau$ has zero mean in $B(0,5)$. Using Theorem \ref{Poincarepesata} (see Remark \ref{zeromeans}), we get 

\begin{equation}\label{peso2}
    \| \tilde w_\tau\|_{L^\nu(B(0,5))}\le C_{\tilde \rho} \int_{B(0,5)}\tilde \rho |\nabla \tilde w_\tau|^2 \le  C_{\tilde \rho}C,
\end{equation}

for every $2<\nu<\overline 2_p.$ 

Let us consider the following test function 
 \begin{equation}\label{filinearizzato}
 \phi=\eta^2\dfrac{1}{\overline v_\tau}(|\tilde{w}_\tau|^{\beta}+(2\beta)^{\beta})\qquad \beta\geq 1,
 \end{equation}
 with $\eta\geq 0$ and $\eta\in C_0^1(B(0,5))$. Therefore
\begin{equation}\label{derifilinearizzato}
\nabla\phi=\dfrac{2\eta}{\overline v_\tau}(|\tilde{w}_\tau|^{\beta}+(2\beta)^{\beta})\nabla\eta +\eta^2\dfrac{1}{\overline v_{\tau}^2}(\beta sign(\tilde{w}_\tau)|\tilde{w}_\tau|^{\beta -1}-|\tilde{w}_\tau|^{\beta}-(2\beta)^{\beta})\nabla \overline v_{\tau}
\end{equation}

We set $\tilde \rho=H^{p-2}(\nabla \tilde u)$. Substituting \eqref{filinearizzato} and \eqref{derifilinearizzato} in \eqref{newlinearizzato} we get

\begin{equation}\label{stimanuovolinearizzato2}
\begin{split}
&I_1+\cdot\cdot\cdot+I_6\\&:=
    (p-1)\int_{\tilde{\Omega}}\tilde \rho \langle \nabla H(\nabla \tilde u),\nabla \overline v \rangle \langle \nabla H(\nabla \tilde u),\nabla \eta \rangle \dfrac{2\eta}{\overline v_\tau}(|\tilde{w}_\tau|^{\beta}+(2\beta)^{\beta})\,dy \\&+
       (p-1)\int_{\tilde{\Omega}}\tilde \rho \langle \nabla H(\nabla \tilde u),\nabla \overline v \rangle \langle \nabla H(\nabla \tilde u),\nabla \overline v_\tau \rangle \eta^2\dfrac{1}{\overline v_{\tau}^2}(\beta sign(\tilde{w}_\tau)|\tilde{w}_\tau|^{\beta -1}-|\tilde{w}_\tau|^{\beta}-(2\beta)^{\beta}) \,dy \\ 
    &+\int_{\tilde \Omega} H^{p-1}(\nabla \tilde u)\langle D^2H(\nabla \tilde u) \nabla \overline v, \nabla \eta \rangle \dfrac{2\eta}{\overline v_\tau}(|\tilde{w}_\tau|^{\beta}+(2\beta)^{\beta}) \,dy
    \\ &+\int_{\tilde \Omega} H^{p-1}(\nabla \tilde u)\langle D^2H(\nabla \tilde u) \nabla \overline v, \nabla \overline v_\tau \rangle \eta^2\dfrac{1}{\overline v_{\tau}^2}(\beta sign(\tilde{w}_\tau)|\tilde{w}_\tau|^{\beta -1}-|\tilde{w}_\tau|^{\beta}-(2\beta)^{\beta})\,dy \\
    &+q \int_{\tilde \Omega} H^{q-1}(\nabla \tilde u) \langle \nabla H(\nabla \tilde u),\nabla \overline v\rangle \overline a(y,\tilde u)\eta^2\dfrac{1}{\overline v_\tau}(|\tilde{w}_\tau|^{\beta}+(2\beta)^{\beta}) \,dy\\
  &- \int_{\tilde \Omega} \overline{c} \overline v \eta^2\dfrac{1}{\overline v_\tau}(|\tilde{w}_\tau|^{\beta}+(2\beta)^{\beta} )\,dy\ge 0.
    \end{split}
\end{equation}

Now we estimate the terms $-(I_2+I_4)$. By \eqref{desi2} and from Lemma \ref{stimedamascellianisotrope} (see \eqref{eq:inequalities3}), we have

\begin{equation}\label{stimepezzi}
    \begin{split}
    -(I_2+I_4)&\ge \tilde C_1\beta\int_{\tilde \Omega} \tilde \rho |\nabla \overline v_\tau|^2 \frac{\eta ^2}{\overline v_\tau^2}|\tilde w_\tau|^{\beta-1}\,dy.
    \\& =\tilde C_1\beta\int_{\tilde \Omega} \tilde \rho |\nabla 
    \tilde w_\tau|^2 \eta ^2|\tilde w_\tau|^{\beta-1}\,dy.
    \end{split}
\end{equation}

From Lemma \ref{stimedamascellianisotrope} (see \eqref{eq:inequalities4}) applied on terms $I_1+I_3$, by \eqref{stimepezzi} and \eqref{eq:BddH}, we have 

\begin{equation}
\begin{split}
    &\tilde C_1\beta\int_{\tilde \Omega} \tilde \rho |\nabla 
    \tilde w_\tau|^2 \eta ^2|\tilde w_\tau|^{\beta-1}\,dy\\
    & \le 2\tilde C_2 \int_{\tilde \Omega} \tilde \rho|\nabla \tilde w_\tau| |\nabla \eta| \eta (|\tilde{w}_\tau|^{\beta}+(2\beta)^{\beta})\,dy \\
   &+C(p,q,H,\delta,\|\tilde u\|_{L^{\infty}(\tilde \Omega)}) \int_{\tilde \Omega}  H^{q-1}(\nabla \tilde u) |\nabla \tilde w_\tau|  \eta ^2 (|\tilde{w}_\tau|^{\beta}+(2\beta)^{\beta}) \,dy\\ 
   & +C(p,q,H,\delta,\|\tilde u\|_{L^{\infty}(\tilde \Omega)},\|\nabla \tilde u\|_{L^{\infty}(\tilde \Omega)})\int_{\tilde \Omega } \eta ^2 (|\tilde{w}_\tau|^{\beta}+(2\beta)^{\beta})\,dy,
    \end{split}
\end{equation}
where $C(p,q,H,\delta,\|\tilde u\|_{L^{\infty}(\tilde \Omega)})$ and $C(p,q,H,\delta,\|\tilde u\|_{L^{\infty}(\tilde \Omega)},\|\nabla \tilde u\|_{L^{\infty}(\tilde \Omega)})$ are positive constants.

From weighted Young's inequality and standard Young's inequality applied to first and second term on the right hand side of the previous inequality, since $q\ge p-1$, we have
\begin{equation}\label{eqtot}
    \begin{split}
    & \tilde C_1\beta\int_{\tilde \Omega} \tilde \rho |\nabla \tilde w_\tau|^2 \eta ^2|\tilde w_\tau|^{\beta-1} \,dy
     \\&\le  \epsilon \int_{\tilde \Omega} \tilde \rho |\nabla \tilde w_\tau|^2 \eta ^2|\tilde w_\tau|^{\beta-1} \,dy+\frac{C}{\epsilon}\int_{\tilde\Omega}\tilde \rho |\tilde w_\tau|^{\beta+1}|\nabla \eta|^2\,dy\\& + C\int_{\tilde\Omega} \tilde \rho (2\beta)^{\beta}|\nabla \eta|^2\,dy
     +C\int_{\tilde \Omega}\tilde \rho |\nabla \tilde w_\tau|^2\eta^2(2\beta)^{\beta}\,dy\\&+\frac{C}{\epsilon}\int_{\tilde \Omega} \tilde \rho\eta^2 |\tilde w_\tau|^{\beta+1}\,dy+C\int_{\tilde\Omega}\tilde \rho \eta^2(2\beta)^{\beta}\,dy
     \\ & +C\int_{\tilde \Omega } \eta ^2 (|\tilde{w}_\tau|^{\beta}+(2\beta)^{\beta})\,dy,
     \end{split}
\end{equation}
with $C=C(p,q,H,\delta,\|\tilde u\|_{L^{\infty}(\tilde \Omega)},\|\nabla \tilde u\|_{L^{\infty}(\tilde \Omega)})$ positive constant.
 
By \eqref{peso11}, for $\epsilon=\tilde C_1\beta/2$, the previous inequality \eqref{eqtot} becomes

\begin{eqnarray}\label{caos114}
&&\beta\int_{\tilde \Omega}\tilde \rho |\nabla \tilde w_{\tau}|^2\eta^2|\tilde{w}_\tau|^{\beta -1}\,dy\\\nonumber
&\leq&C\int_{\tilde \Omega}\tilde \rho(|\tilde{w}_\tau |^{\beta +1}+(2\beta)^{\beta})|\nabla\eta|^2\,dy+C\int_{\tilde \Omega}\eta^2(|\tilde{w}_\tau|^{\beta}+(2\beta)^{\beta}) \,dy\\\nonumber
&+&C\int_{\tilde\Omega}\tilde \rho\eta^2|\tilde{w}_\tau|^{\beta+1}\,dy+C(2\beta)^{\beta}.
\end{eqnarray}

By \eqref{caos114}, arguing as in the proof of Theorem \ref{t1}, we have 
\begin{equation}\label{barrafael14}
\int_{\tilde \Omega}\tilde \rho|\nabla \tilde w_{\tau}|^2\eta^2|\tilde{w}_\tau|^{\beta -1}\,dy\leq C\int_{\tilde \Omega}(|\tilde{w}_\tau|^{\beta +1}+(2\beta)^{\beta})(\eta^2+\tilde \rho |\nabla\eta|^2)\,dy.
\end{equation}
By \eqref{barrafael14}, it follows the estimate \eqref{secondstima}.

Hence the Moser’s iterative method applies. Therefore we get that there exist a positive constant $C$ such that 

\begin{equation}
    \|\overline v\|_{L^s(B(0,2))}\le C \inf_{B(0,1)} \overline v.
\end{equation}

Finally, scaling back to the coordinates $x=x_0+\delta y$ and recalling that $\overline{v}:=\tilde v+ \|\overline g\|_{L^{\infty}}$ we get 

\begin{equation*}
\delta^{-\frac{N}{s}} \|v\|_{L^s(B(x_0,2\delta))}\le C \left(\inf_{B(x_0,\delta)} v + \|g\|_{L^{\infty}(B(x_0,6\delta))}\right).
\end{equation*}

\item  {Case $(b):{(2N+2)}/{(N+2)}<p<2$.}

Arguiung exactly as in case $(a)$, we are able to get $\eqref{stimaimportante}$. Using \eqref{deinizionewtaulinearizzato}, we still obtain \eqref{quintaequazionehy} if $\beta \neq 1$ or \eqref{betaugualemeno11} if $\beta =-1$.

Using the classical Sobolev inequality (see \eqref{normab1}) we get
\begin{equation}\label{sobolevclassicahh}
\|\eta \tilde{w}_{\tau}\|^2_{L^{2^*}(\tilde \Omega)}\le  \hat C r^2\int_{\tilde \Omega} \tilde w_\tau^2(\eta^2+\tilde \rho|\nabla \eta|^2)\,dy.
\end{equation}
where $2^*$ is the classical Sobolev exponent

Now, from Theorem \ref{local1} (in particular see \eqref{stima peso locale}), we have that $\tilde \rho \in L^t(\tilde \Omega)$, with $\\t:=(p-1)/(2-p)\theta$, for $0<\theta<1$. Therefore 
\begin{equation}\label{cosecosehh}
    \|\eta \tilde{w}_{\tau}\|^2_{L^{2^*}(\tilde \Omega)}\le  \hat C r^2\|\tilde \rho\|_{L^t(\tilde \Omega)}\|\tilde w_\tau(\eta+|\nabla \eta|)\|_{L^{t^\#}(\tilde \Omega)},
\end{equation}

where $t^\#=2t/(t-1)$. Now we note that $\chi '=2^*/t^\#>1$, for ${(2N+2)}/{(N+2)}<p<2$. Then also in this case the Moser’s iterative method works and scaling back we get the thesis.

\end{proof}

\begin{rem}\label{remarkino}
    In the case $a_{x_i}\le 0$ and  $f_{x_i}\ge 0$, for any $i=1,...,n$, we have 
\begin{equation}\label{Harnacklinearizz}
\|v\|_{L^s(B(x_0,2\delta))}\le C \inf_{B(x_0,\delta)} v. 
\end{equation}
Indeed, in this case, the last term in \eqref{supersol} is nonpositive and therefore the introduction of the function $\overline g$ is unnecessary.
\end{rem}

If we consider a subsolution $v$ of \eqref{soluzionelinearizzato}, the iteration technique, which is simpler in this case, allows to prove the following  

\begin{thm}\label{Harnacklinearizzato2}
   Let $p\ge 2$ and $u\in C^{1,\alpha}_{loc}(\Omega)$ be a weak solution to~\eqref{eq:Euler-Lagrange}. Assume that
$ q\geq\max\,\{p-1,1\}$ and suppose that $H$, $f(x,s), a(x,s)$ fulfill $(h_H)$ and $(h_p)$.
Assume that $\overline {B(x_0,6\delta)}\subset \Omega$.
If $v\in H^{1,2}_{\rho,loc}(\Omega)\cap L^{\infty}_{loc}(\Omega)$ is a nonnegative subsolution of \eqref{soluzionelinearizzato}, then for all $s>1$, there exist $C=C(p,q,H,\delta,\| u\|_{L^{\infty}_{loc}( \Omega)},\|\nabla  u\|_{L^{\infty}_{loc}( \Omega)})>0$ such that 

\begin{equation}\label{Harnacklinearperilsup}
\sup_{B(x_0,\delta)} v\le C \left(\delta^{-\frac{N}{s}} \|v\|_{L^s(B(x_0,2\delta))} + \|g\|_{L^{\infty}(B(x_0,6\delta))}\right)
\end{equation}

where $g:=\delta ^{p}( f_{x_i}(x, u)- a_{x_i}(x,u)H^q(\nabla u))$.

If $(2N+2)/(N+2)<p<2$ the same result holds for $s>2^*/\bar t^\sharp$,
where $$\frac{1}{\bar t^\sharp}=\frac{2p-3}{2(p-1)}$$
and  $2^*$ is  the classical Sobolev exponent $2^*=2N/(N-2)$.
\end{thm}

\begin{proof}
As in the proof of the Theorem \ref{Harnacklinearizzato}, we set $y:=(x-x_o)/\delta$ and $\tilde \Omega:=(\Omega-x_0)/\delta$. Rescaling the linearized operator \eqref{linearizzato},  we have (see \eqref{newlinearizzato})

\begin{equation}\label{newlinearizzatosup}
\begin{split}
    (p-1)&\int_{\tilde{\Omega}} H^{p-2}(\nabla \tilde u)\langle \nabla H(\nabla \tilde u),\nabla \overline v \rangle \langle \nabla H(\nabla \tilde u),\nabla \varphi \rangle \,dy \\ 
    +&\int_{\tilde \Omega} H^{p-1}(\nabla \tilde u)\langle D^2H(\nabla \tilde u) \nabla \overline v, \nabla \varphi \rangle  \,dy +q \int_{\tilde \Omega} H^{q-1}(\nabla \tilde u) \langle \nabla H(\nabla \tilde u),\nabla \overline v\rangle \overline a(y,\tilde u)\varphi \,dy \\
  - &\int_{\tilde \Omega} \overline{c} \overline v \varphi \,dy \le 0.
    \end{split}
\end{equation}

Now, we set $\overline v_\tau:=\overline v+\tau$, where $\tau$ is a positive constant. Let us consider the test function $\varphi=\eta^2\overline v_\tau^\beta$, where $\beta >0$ and $0\le \eta\in C^1_c(B(0,6))$.

For $p\ge 2$, using $\varphi$ in \eqref{newlinearizzatosup}, by performing the same computations as in the proof of Theorem \ref{Harnacklinearizzato}, we get (see \eqref{normab1})
    \begin{equation}
\label{normab114}
\|\eta \tilde{w}_{\tau}\|^2_{L^{\nu}(\tilde \Omega)}\le  \hat C r^2\|\tilde{w}_{\tau}(\eta+|\nabla\eta|)\|^2_{L^2(\tilde \Omega)}.
\end{equation}

A standard application of Moser's iteration provides
\begin{equation*}
    \sup_{B(0,1)} \overline v\le C\|\overline v\|_{L^s(B(0,2))} .
\end{equation*}

Finally, scaling back to the coordinates $x=x_0+\delta y$ and recalling that $\overline{v}:=\tilde v+ \|\overline g\|_{L^{\infty}}$ we get 

\begin{equation*}
\sup_{B(x_0,\delta)} v\le C \left(\delta^{-\frac{N}{s}} \|v\|_{L^s(B(x_0,2\delta))} + \|g\|_{L^{\infty}(B(x_0,6\delta))}\right),
\end{equation*}
for all $s>1$.
For $(2N+2)/(N+2)<p<2$, using $\varphi$ as test function in \eqref{newlinearizzatosup}, as in the proof of the Theorem \ref{Harnacklinearizzato}, we get (see \eqref{cosecosehh})

\begin{equation*}
    \|\eta \tilde{w}_{\tau}\|^2_{L^{2^*}(\tilde \Omega)}\le  \hat C r^2\|\tilde \rho\|_{L^t(\tilde \Omega)}\|\tilde w_\tau(\eta+|\nabla \eta|)\|_{L^{t^\#}(\tilde \Omega)},
\end{equation*}

where $t^\#=2(p-1)/(2p-3)$.

A Moser iteration combined with a rescaling gives
\begin{equation*}
\sup_{B(x_0,\delta)} v\le C \left(\delta^{-\frac{N}{s}} \|v\|_{L^s(B(x_0,2\delta))} + \|g\|_{L^{\infty}(B(x_0,6\delta))}\right).
\end{equation*}
\end{proof}

At this stage, the proof of Theorem \ref{Harnacklinearizzo} follows directly.

\begin{proof}[Proof of Theorem \ref{Harnacklinearizzo}]
    The proof follows directly by Theorem \ref{Harnacklinearizzato} and Theorem \ref{Harnacklinearizzato2}.
\end{proof}

As a direct consequence of the Harnack inequality, we obtain the proof of Theorem \ref{Strong2}.

\begin{proof}[Proof of Theorem \ref{Strong2}]
    Let us define the set $$U_v=\{x\in \Omega':v(x)=0\}.$$ By the continuity of $v$, the set $U_v$ is a closed set in $\Omega'$. On the other hand, since $a_{x_i}\le 0$ and $f_{x_i}\ge 0$ for any $i=1,...,n$, from Theorem \ref{Harnacklinearizzo} (see Remark \ref{remarkino}) it follows that $U_v$ is open. Then the thesis follows.
\end{proof}

\begin{center}{\bf Acknowledgements}\end{center}  
  D. Vuono is partially supported by PRIN project P2022YFAJH\_003 (Italy): Linear and nonlinear PDEs; new directions and applications. 
  D. Vuono is partially supported also by Gruppo Nazionale per l’Analisi
 Matematica, la Probabilità e le loro Applicazioni (GNAMPA) of the Istituto Nazionale
 di Alta Matematica (INdAM).
\		
\begin{center}
 {\sc Data availability statement}\
All data generated or analyzed during this study are included in this published article.
\end{center}

   \end{document}